\documentclass[10pt,a4paper,reqno]{amsart}
\usepackage{amsthm}
\usepackage{amsmath}
\usepackage{amssymb}
\usepackage{graphicx,color}

\usepackage[font=small]{caption}

\usepackage{multirow}
\usepackage[shortlabels]{enumitem}

\makeatletter
\theoremstyle{plain}
\newtheorem{thm}{Theorem}
  \theoremstyle{definition}
  
  \theoremstyle{remark}
  \newtheorem{rem}[thm]{Remark}
  \theoremstyle{plain}
  
  \theoremstyle{plain}
  \newtheorem{lem}[thm]{Lemma}
  \theoremstyle{plain}
  \newtheorem{cor}[thm]{Corollary}
 \theoremstyle{definition}
  
  \theoremstyle{remark}
  \newtheorem*{rem*}{Remark}

  \theoremstyle{definition}

\usepackage{amsfonts}
\usepackage{mathrsfs}

\newcommand{\N}{\mathbb{N}}
\newcommand{\R}{{\mathbb{R}}}
\newcommand{\C}{{\mathbb{C}}}

\newcommand{\dd}{{\rm d}}
\newcommand{\ii}{{\rm i}}

\newcommand{\diag}{\mathop\mathrm{diag}\nolimits}
\newcommand{\spn}{\mathop\mathrm{span}\nolimits}
\newcommand{\Dom}{\mathop\mathrm{Dom}\nolimits}

\newcommand{\Ran}{\mathop\mathrm{Ran}\nolimits}
\newcommand{\Ker}{\mathop\mathrm{Ker}\nolimits}
\newcommand{\spec}{\mathop\mathrm{spec}\nolimits}

\renewcommand{\Re}{\mathop\mathrm{Re}\nolimits}

\newcommand{\Tr}{\mathop\mathrm{Tr}\nolimits}

\newcommand*\pFqskip{8mu} % Macro for hypergeometric series
\catcode`,\active
\newcommand*\pFq{\begingroup
        \catcode`\,\active
        \def ,{\mskip\pFqskip\relax}%
        \dopFq
}
\catcode`\,12
\def\dopFq#1#2#3#4#5{%
        {}_{#1}F_{#2}\biggl(\genfrac..{0pt}{}{#3}{#4}\biggl|\,#5\biggr)%
        \endgroup
}

\makeatother

\begin{document}

\title[]{Spectral representation of some weighted Hankel matrices and orthogonal polynomials from the Askey scheme}

\author{Franti{\v s}ek {\v S}tampach}
\address[Franti{\v s}ek {\v S}tampach]{
	Department of Applied Mathematics, Faculty of Information Technology, Czech Technical University in~Prague, 
	Th{\' a}kurova~9, 160~00 Praha, Czech Republic
	}
%\curraddr{}
\email{stampfra@fit.cvut.cz}

\author{Pavel {\v S}{\v t}ov{\' i}{\v c}ek}
\address[Pavel {\v S}{\v t}ov{\' i}{\v c}ek]{
	Department of Mathematics, Faculty of Nuclear Science and Physical Engineering, Czech Technical University in~Prague,
	Trojanova~13, 120~00 Praha, Czech Republic}
\email{stovicek@fjfi.cvut.cz}

\subjclass[2010]{47B35; 47B37; 47A10; 33C45}

\keywords{weighted Hankel matrix, spectral representation, diagonalization, orthogonal polynomials, Askey scheme}

\date{\today}

\begin{abstract}
We provide an explicit spectral representation for several weighted Hankel matrices by means of the so called commutator method. These weighted Hankel matrices are found in the commutant of Jacobi matrices associated with orthogonal polynomials from the Askey scheme whose Jacobi parameters are polynomial functions of the index. We also present two more results of general interest. First, we give a complete description of the commutant of a Jacobi matrix. Second, we deduce a necessary and sufficient condition for a weighted
Hankel matrix commuting with a Jacobi matrix to determine a unique self-adjoint operator on~$\ell^{2}(\N_{0})$.
\end{abstract}

\maketitle

\section{Introduction}

One of the very few Hankel matrices that admits a completely explicit diagonalization is the
Hilbert matrix whose entries are given by $1/(m+n+1)$, for $m,n\in\N_{0}$.
The Hilbert matrix represents a bounded operator acting on $\ell^{2}(\N_{0})$ with continuous spectrum 
equal to the interval $[0,\pi]$. This was shown by Magnus~\cite{magnus50}, see also~\cite{magnus49, shanker49}
for related results. Later on, Rosenblum used an integral operator unitarily equivalent to the generalized Hilbert matrix -- a one-parameter generalization of the Hilbert matrix -- and applied an idea of the commutator method to deduce a complete spectral representation of the generalized Hilbert matrix, see~\cite{rosenblum2}.

The transformation of the Hilbert matrix in an integral operator is not necessary and one can apply the 
commutator method directly to the semi-infinite matrices. This approach relies on finding a Jacobi matrix with explicitly solvable spectral problem that commutes with the Hilbert matrix (or any other operator whose spectral properties are to be analyzed). In the case of the generalized Hilbert matrix, such Jacobi matrix exists indeed~\cite{kalvoda-stovicek_lma16} and corresponds to a~special family of orthogonal polynomials known as the continuous dual Hahn polynomials~\cite[Sec.~9.3]{koekoek-etal_10}. A continuous variant of the commutator method where the role of the Jacobi operator and the Hankel matrix is played by a second order differential operator and an integral Hankel operator, respectively, was described by Yafaev in~\cite{yafaev_faa10} and applied to several concrete Hankel integral operators therein.

In fact, the authors of~\cite{kalvoda-stovicek_lma16} used the commutator method to diagonalize a three parameter family of \emph{weighted} Hankel matrices which reduces to the Hilbert matrix in a very special case. By the weighted Hankel matrix, we mean a matrix with entries of the form
\begin{equation}
 \mathcal{H}_{m,n}:=w_{m}w_{n}h_{m+n}, \quad m,n\in\N_{0},
\label{eq:def_weighted_hankel}
\end{equation}
where $w_{n}\neq0$, $\forall n\in\N_{0}$, are referred to as weights. In contrast to the spectral theory
of the Hankel operators which is deeply developed nowadays~\cite{bottcher-silbermann_06,peller_03,widom_tams66},
the general theory of their weighted generalization is almost missing and various results appear rather sporadically, 
see, for example, Chp.~6,~\S8 in~\cite{peller_03}.

In this paper, $\mathcal{J}$ stands for a Hermitian non-decomposable Jacobi matrix, where for $n\in\N_{0}$, we denote
\[
 b_{n}:=\mathcal{J}_{n,n} \quad\mbox{ and }\quad a_{n}:=\mathcal{J}_{n,n+1}=\mathcal{J}_{n+1,n}.
\]
The assumption requiring $\mathcal{J}$ to be non-decomposable means that $a_{n}\neq0$, $\forall n\in\N_{0}$.
We focus on the class of Jacobi matrices where $b_{n}$ and $a_{n}^{2}$ are polynomials in $n$.
As $a^{2}_{n}$ is an analytic function of $n$, we add an extra condition requiring $a_{-1}=0$. We seek for weighted Hankel matrices commuting with such $\mathcal{J}$. We continue the analysis started in~\cite{kalvoda-stovicek_lma16} and use the Askey scheme of hypergeometric orthogonal polynomials~\cite{koekoek-etal_10}
as a rich source of Jacobi operators with explicitly solvable spectral problem. Except the aforementioned continuous dual Hahn polynomials, among the orthogonal polynomials from the Askey scheme whose Jacobi parameters $b_{n}$ and $a_{n}^{2}$ are polynomials of $n$ there belong Hermite, Laguerre, Meixner, and Meixner--Pollaczek polynomials. In addition, Krawtchouk and dual Hahn polynomials whose measures of orthogonality are finitely supported also fulfill the restriction in the setting of finite Jacobi matrices.

Here and below, by the commutant of~$\mathcal{J}$ we mean the space of matrices formally commuting with~$\mathcal{J}$. It can be shown that there is always a weighted Hankel matrix in the commutant of~$\mathcal{J}$ if $b_{n}$ is a polynomial of~$n$ of degree less than or equal to~$2$ and $a_{n}^{2}$ a polynomial of~$n$ of degree less than or equal to~$4$. However, we do not discuss scrupulously all the configurations for the degrees. Instead, we focus 
only on several configurations corresponding to the families of orthogonal polynomials mentioned above 
for which the Jacobi matrix is explicitly diagonalizable and the commutator method yields the spectral 
representation of an obtained weighted Hankel matrix. Furthermore, we restrict the obtained class of weighted 
Hankel matrices to those which determine a densely defined operator on $\ell^{2}(\N_{0})$ and does not degenerate to the case of a~rank-one operator whose matrix entries are of the form $\mathcal{H}_{m,n}=w_{n}w_{m}$, for $m,n\in\N_{0}$, because their spectral properties are not interesting. 

The paper is organized as follows. Section~\ref{sec:main} contains a summary of main results, i.e., explicit spectral representations 
of several operators given by weighted Hankel matrices. In Section~\ref{sec:prelim}, preliminary but necessary results are worked out.
First of all, Theorem~\ref{thm:commut_matr_Jac} provides a description of the commutant of a Jacobi matrix in terms of corresponding orthonormal
polynomials. Second, Theorem~\ref{thm:def_H} provides a necessary and sufficient condition on a weighted Hankel matrix from the commutant of a Jacobi matrix
to determine a unique self-adjoint operator on~$\ell^{2}(\N_{0})$. Regardless of the main aim, both theorems can be of independent interest.
Section~\ref{sec:proofs} contains proofs of the main results from Section~\ref{sec:main}. 
Finally, in Section~\ref{sec:conseq_spec_func}, several identities for orthogonal polynomials that follow from the main results as consequences are given.
Some of the formulas are known, some seem to be new.

\section{Main results}\label{sec:main}

In this section we summarize main results providing explicit spectral representations for several operators determined by special weighted Hankel matrices. By a spectral representation of a given self--adjoint operator $H$ on a Hilbert space $\mathscr{H}$ we understand a triplet $(T_{h},L^{2}(\R,\dd\mu),U)$, where $T_{h}$ is an operator of multiplication by a real function $h$ on $L^{2}(\R,\dd\mu)$ unitarily equivalent to~$H$ via a unitary mapping $U:\mathscr{H}\to L^{2}(\R,\dd\mu)$, i.e., $UHU^{-1}=T_{h}$. We specify the function $h$, the measure $\mu$, and the unitary mapping $U$ in theorems below and draw immediate conclusions on the spectrum of $H$. 

Here $\mathscr{H}=\ell^{2}(\N_{0})$ and $\{e_{n} \mid n\in\N_{0}\}$ denotes the standard basis of~$\ell^{2}(\N_{0})$. In each of the studied cases, the unitary mapping $U:\ell^{2}(\N_{0})\to L^{2}(\R,\dd\mu)$ is defined by its values on the standard basis, i.e.,
\begin{equation}
 U:e_{n}\mapsto\hat{P}_{n}, \quad \forall n\in\N_{0},
 \label{eq:def_U}
\end{equation}
where $\{\hat{P}_{n} \mid n\in\N_{0}\}$ stands for an orthonormal basis of $L^{2}(\R,\dd\mu)$ given by the normalized polynomials orthogonal with respect to~$\mu$. Furthermore, columns as well as rows of the weighted Hankel matrix~$\mathcal{H}$ in question belong to~$\ell^{2}(\N_{0})$, $\mathcal{H}$ determines unambiguously a~self-adjoint operator~$H$ on~$\ell^{2}(\N_{0})$, and the unitary equivalence $UHU^{-1}=T_{h}$ implies that
\[
 \sum_{n=0}^{\infty}\mathcal{H}_{m,n}\hat{P}_{n}=h\hat{P}_{m}, \quad \forall m\in\N_{0},
\]
where the left-hand side converges in~$L^{2}(\R,\dd\mu)$.

The first weighted Hankel matrix $\mathcal{H}^{(1)}$ has entries
\begin{equation}
 \mathcal{H}^{(1)}_{m,n}:=\frac{k^{m+n}\Gamma(m+n+\alpha)}{\sqrt{m!n!\Gamma(m+\alpha)\Gamma(n+\alpha)}}, \quad m,n\in\N_{0},
\label{eq:def_hank1}
\end{equation}
with $k\in(0,1)$ and $\alpha>0$. The corresponding operator on $\ell^{2}(\N_{0})$ possesses an interesting
spectral phase transition depending on the value of the parameter $k$. The diagonalization of this operator will be carried out with the aid of Meixner--Polaczek, Laguerre, and Meixner polynomials depending on the value of the parameter $k$.
Therefore we recall the definition of the Meixner--Polaczek polynomials~\cite[Eq.~9.7.1]{koekoek-etal_10}
\begin{equation}
P_{n}^{(\lambda)}(x;\phi)=\frac{(2\lambda)_{n}}{n!}\,\, e^{\ii n\phi}
\pFq{2}{1}{-n,\lambda+\ii x}{2\lambda}{1-e^{-2\ii\phi}}, \quad \phi\in(0,\pi), \; \lambda>0,
\label{eq:def_meix-polllac}
\end{equation}
the Laguerre polynomials~\cite[Eq.~9.12.1]{koekoek-etal_10}
\begin{equation}
L_{n}^{(a)}(x)=\frac{(a+1)_{n}}{n!}\pFq{1}{1}{-n}{a+1}{x}, \quad a>-1,
\label{eq:def_laguerre}
\end{equation}
and the Meixner polynomials~\cite[Eq.~9.10.1]{koekoek-etal_10}
\begin{equation}
 M_{n}(x;\beta,c)=\pFq{2}{1}{-n,-x}{\beta}{1-\frac{1}{c}}, \quad c\in(0,1), \; \beta>0.
\label{eq:def_meixner}
\end{equation}
We follow the standard notation for the hypergeometric series, see for example~\cite{koekoek-etal_10} for definitions.

\begin{thm}\label{thm:meixpol_lag_meix}
 For $k\in(0,1)$ and $\alpha>0$, the matrix $\mathcal{H}^{(1)}$ given by~\eqref{eq:def_hank1} determines a unique self-adjoint operator $H^{(1)}$ on $\ell^{2}(\N_{0})$
 which is unitarily equivalent via $U$ given by~\eqref{eq:def_U} to the operator of multiplication by $h$ on $L^{2}(\R,\dd\mu)$ that are as follows:
 \begin{enumerate}[{\upshape i)}]
  \item If $1/2<k<1$, then $\mu$ is absolutely continuous supported on $\R$ with the density 
  \[
   \frac{\dd\mu}{\dd x}(x)=\frac{(4k^{2}-1)^{(\alpha-1)/2}}{2\pi k^{\alpha}\Gamma(\alpha)}
   \exp\!\left(\frac{-2x}{\sqrt{4k^{2}-1}}\arcsin\left(\frac{1}{2k}\right)\!\right)
   \!\left|\Gamma\left(\frac{\alpha}{2}+\frac{\ii x}{\sqrt{4k^{2}-1}}\right)\right|^{2}\!,
  \]
  \[
    h(x)=k^{-\alpha}\exp\!\left(\frac{2x}{\sqrt{4k^{2}-1}}\arccos\left(\frac{1}{2k}\right)\!\right)\!,
  \]
  and
  \[
   \hat{P}_{n}(x)=\sqrt{\frac{\Gamma(\alpha)n!}{\Gamma(n+\alpha)}}\,P_{n}^{(\alpha/2)}\!\left(\frac{x}{\sqrt{4k^{2}-1}};\arccos\left(\frac{1}{2k}\right)\!\right)\!,
  \]
  for $x\in\R$. Consequently, $H^{(1)}$ is an unbounded operator with
  \[
   \spec\left(H^{(1)}\right)=\spec_{\text{ac}}\left(H^{(1)}\right)=[\,0,+\infty).
  \]
  \item If $k=1/2$, then $\mu$ is absolutely continuous supported on $[0,\infty)$ with the density  
  \[
   \frac{\dd\mu}{\dd x}(x)=\frac{2^{\alpha}}{\Gamma(\alpha)}x^{\alpha-1}e^{-2x},
  \]
  \[
  h(x)=2^{\alpha}e^{-2x},
  \]
  and
  \[
  \hat{P}_{n}(x):=\sqrt{\frac{\Gamma(\alpha)n!}{\Gamma(n+\alpha)}}\,L^{(\alpha-1)}_{n}(2x),
  \]
  for $x\in(0,\infty)$. Consequently, $H^{(1)}$ is a bounded operator with
  \[
  \spec\left(H^{(1)}\right)=\spec_{\text{ac}}\left(H^{(1)}\right)=[0,2^{\alpha}]\quad\mbox{ and }\quad \|H\|=2^{\alpha}.
  \]
  \item If $0<k<1/2$, then $\mu$ is discrete supported on $\sqrt{1-4k^{2}}\,\N_{0}$ such that, for $n\in\N_{0}$,
  \[
   \mu\!\left(n\sqrt{1-4k^{2}}\right)=
   \frac{(\alpha)_{n}}{n!(2k^{2})^{n+\alpha}}\left(4k^{2}-1+\sqrt{1-4k^{2}}\right)^{\!\alpha}\left(1-2k^{2}-\sqrt{1-4k^{2}}\right)^{\!n}.
  \]
  Further, for $x\in\sqrt{1-4k^{2}}\,\N_{0}$, we have
  \[
   h(x)=\left(\frac{1-\sqrt{1-4k^{2}}}{2k^{2}}\right)^{\!\alpha}\left(\frac{1-2k^{2}-\sqrt{1-4k^{2}}}{2k^{2}}\right)^{x/\sqrt{1-4k^{2}}}
  \]
  and
  \begin{align*}
   \hat{P}_{n}(x)=&\left(\frac{1-2k^{2}-\sqrt{1-4k^{2}}}{2k^{2}}\right)^{\!n/2}\\
   &\hskip12pt\times\sqrt{\frac{\Gamma(n+\alpha)}{\Gamma(\alpha)n!}}\,M_{n}\left(\frac{x}{\sqrt{1-4k^{2}}};\alpha,\frac{1-2k^{2}-\sqrt{1-4k^{2}}}{2k^{2}}\right)\!.
  \end{align*}
  Consequently, $H^{(1)}$ is a compact (even trace class) operator with 
  \[
  \spec\left(H^{(1)}\right)\setminus\{0\}=\spec_{\text{p}}\left(H^{(1)}\right)=
  \left(\frac{1-\sqrt{1-4k^{2}}}{2k^{2}}\right)^{\!\alpha}\!\left(\frac{1-2k^{2}-\sqrt{1-4k^{2}}}{2k^{2}}\right)^{\!\N_{0}}
  \]
  and
  \[
   \|H^{(1)}\|=\left(\frac{1-\sqrt{1-4k^{2}}}{2k^{2}}\right)^{\!\alpha}.
  \]
 \end{enumerate}
\end{thm}

The second weighted Hankel matrix $\mathcal{H}^{(2)}$ is given by formulas
\begin{equation}
 \mathcal{H}_{m,n}^{(2)}=(-1)^{m(m-1)/2+n(n-1)/2}\sqrt{\frac{\Gamma(m+2\lambda)\Gamma(n+2\lambda)}{m!n!}}
 \frac{\Gamma\left((m+n+1)/2\right)}{\Gamma\left(2\lambda+(m+n+1)/2\right)}, 
\label{eq:def_hank2}
\end{equation}
if $m+n$ is even and $\mathcal{H}_{m,n}^{(2)}=0$, if $m+n$ is odd. The parameter $\lambda$ is assumed to be positive. The diagonalization is carried out with the aid of a particular case of the Meixner--Pollaczek polynomials~\eqref{eq:def_meix-polllac}.

\begin{thm}\label{thm:meixpol}
 For $\lambda>0$, the matrix $\mathcal{H}^{(2)}$ given by~\eqref{eq:def_hank2} determines a unique self-adjoint operator $H^{(2)}$ on $\ell^{2}(\N_{0})$
 which is unitarily equivalent via $U$ given by~\eqref{eq:def_U} to the operator of multiplication by $h$ on $L^{2}(\R,\dd\mu)$, where $\mu$ is absolutely continuous
 supported on $\R$ with the density
  \[
   \frac{\dd\mu}{\dd x}(x)=\frac{2^{2\lambda-1}}{\pi\Gamma(2\lambda)}\left|\Gamma(\lambda+\ii x)\right|^{2},
  \]
  \[
  h(x)=\frac{2^{2\lambda-1}}{\Gamma(2\lambda)}|\Gamma(\lambda+\ii x)|^{2},
  \]
  and
  \[
   \hat{P}_{n}(x)=\left(\frac{\Gamma(2\lambda)\,n!}{\Gamma(n+2\lambda)}\right)^{\!1/2}P_{n}^{(\lambda)}\left(x;\frac{\pi}{2}\right)\!,
  \]
  for $x\in\R$. Consequently, $H^{(2)}$ is a bounded operator with
  \[
   \spec\left(H^{(2)}\right)=\spec_{\text{ac}}\left(H^{(2)}\right)=\Bigg[0,\frac{\sqrt{\pi}\,\Gamma(\lambda)}{\Gamma\left(\lambda+1/2\right)}\Bigg] \quad \mbox{ and } 
   \quad \|H^{(2)}\|=\frac{\sqrt{\pi}\,\Gamma(\lambda)}{\Gamma\left(\lambda+1/2\right)}.
  \]
\end{thm}

Since $\mathcal{H}_{m,n}^{(2)}=0$, if $m+n$ odd, we get a corollary on the spectral representation of operators given by weighted Hankel matrices
\begin{equation}
 \mathcal{H}_{m,n}^{(2\text{e})}:=\mathcal{H}_{2m,2n}^{(2)}=(-1)^{m+n}\sqrt{\frac{\Gamma(2m+2\lambda)\Gamma(2n+2\lambda)}{(2m)!(2n)!}}
 \frac{\Gamma\left(m+n+1/2\right)}{\Gamma\left(2\lambda+m+n+1/2\right)}
\label{eq:def_hank2e}
\end{equation}
and
\begin{equation}
 \mathcal{H}_{m,n}^{(2\text{o})}:=\mathcal{H}_{2m+1,2n+1}^{(2)}=(-1)^{m+n}\sqrt{\frac{\Gamma(2m+1+2\lambda)\Gamma(2n+1+2\lambda)}{(2m+1)!(2n+1)!}}
 \frac{\Gamma\left(m+n+3/2\right)}{\Gamma\left(2\lambda+m+n+3/2\right)},
\label{eq:def_hank2o}
\end{equation}
where $m,n\in\N_{0}$ and $\lambda>0$.

\begin{cor}\label{cor:hank2eo}
  For $\lambda>0$, both matrices $\mathcal{H}^{(2\text{e})}$ and $\mathcal{H}^{(2\text{o})}$ given by~\eqref{eq:def_hank2e} and~\eqref{eq:def_hank2o} determine 
  unique self-adjoint operators $H^{(2e)}$ and $H^{(2o)}$ on $\ell^{2}(\N_{0})$ which are unitarily equivalent via mappings $U_{e}$ and $U_{o}$ to the multiplication operator by function~$h$ on $L^{2}((0,\infty),\dd\mu)$, respectively. Here $h$ and $\mu$ are given by the same formulas as in Theorem~\ref{thm:meixpol}, while the unitary operators are determined by their values on the standard basis as 
  \[
   U_{e}:\ell^{2}(\N_{0})\to L^{2}((0,\infty),\dd\mu): e_{n}\mapsto\sqrt{2}\hat{P}_{2n}
  \]
  and
  \[
   U_{o}:\ell^{2}(\N_{0})\to L^{2}((0,\infty),\dd\mu): e_{n}\mapsto\sqrt{2}\hat{P}_{2n+1},
  \]
  where $\hat{P}_{n}$ is again the same as in Theorem~\ref{thm:meixpol}.
\end{cor}

\begin{rem}
 If desirable, the alternating terms in~\eqref{eq:def_hank2e} and~\eqref{eq:def_hank2o} can be omitted since two operators with matrix entries
 $\mathcal{A}_{m,n}$ and $(-1)^{m+n}\mathcal{A}_{m,n}$ are unitarily equivalent via $U:e_{n}\mapsto(-1)^{n}e_{n}$, $n\in\N_{0}$. Further, note that, if $\lambda=1/2$,
 the matrix $\mathcal{H}_{m,n}^{(2\text{e})}$ is a Hankel matrix. In fact, Corollary~\ref{cor:hank2eo} implies that the operator given by the Hankel matrix
 with elements
 \[
  (-1)^{m+n}\mathcal{H}_{m,n}^{(2\text{e})}=\frac{1}{m+n+1/2}, \quad m,n\in\N_{0},
 \]
 has absolutely continuous spectrum equal to $[0,\pi]$. This is a well-know result for the particular case of the generalized Hilbert matrix, see~\cite{rosenblum2} or~\cite{kalvoda-stovicek_lma16}.
 Moreover, the formula for $h$ can be expressed as
 \[
  h(x)=\frac{\pi}{\cosh(\pi x)},
 \]
 which follows from the identity~\cite[Eq.~5.4.4]{dlmf}
 \[
  \left|\Gamma\left(\frac{1}{2}+\ii x\right)\right|^{2}=\frac{\pi}{\cosh(\pi x)}.
 \]
\end{rem}

The third weighted Hankel matrix has elements
\begin{equation}
 \mathcal{H}_{m,n}^{(3)}=(-1)^{m(m-1)/2+n(n-1)/2}\frac{1}{\sqrt{m!n!}}\,\Gamma\!\left(\frac{m+n+1}{2}\right)\!,
\label{eq:def_hank3}
\end{equation}
if $m+n$ is even and $\mathcal{H}_{m,n}^{(3)}=0$, if $m+n$ is odd. An essential role in the diagonalization of the respective operator is played by the Hermite polynomials~\cite[Eq.~9.15.1]{koekoek-etal_10}
\[
 H_{n}(x)=(2x)^{n}\pFq{2}{0}{-n/2,-(n-1)/2}{-}{-\frac{1}{x^{2}}}.
\]

\begin{thm}\label{thm:hermite}
 The matrix $\mathcal{H}^{(3)}$ given by~\eqref{eq:def_hank3} determines a unique self-adjoint operator $H^{(3)}$ on $\ell^{2}(\N_{0})$
 which is unitarily equivalent via $U$ given by~\eqref{eq:def_U} to the operator of multiplication by $h$ on $L^{2}(\R,\dd\mu)$, where $\mu$ is absolutely continuous supported on $\R$ with the density
  \[
   \frac{\dd\mu}{\dd x}(x)=\frac{1}{\sqrt{\pi}}e^{-x^{2}},
  \]
  \[
  h(x)=\sqrt{2\pi}e^{-x^{2}},
   \]
  and
  \[
   \hat{P}_{n}(x)=\frac{1}{\sqrt{2^{n}n!}}H_{n}(x),
  \]
  for $x\in\R$. Consequently, $H^{(3)}$ is a bounded operator with
  \[
   \spec\left(H^{(3)}\right)=\spec_{\text{ac}}\left(H^{(3)}\right)=[\,0,\sqrt{2\pi}\,] \quad \mbox{ and } \quad \|H^{(3)}\|=\sqrt{2\pi}\,.
  \]
\end{thm}

Similarly as in Corollary~\ref{cor:hank2eo} we can use the result of Theorem~\ref{thm:hermite} to obtain the spectral representation of the operators with matrix entries 
\begin{equation}
 \mathcal{H}_{m,n}^{(3\text{e})}:=\mathcal{H}_{2m,2n}^{(3)}=(-1)^{m+n}\frac{\Gamma(m+n+1/2)}{\sqrt{(2m)!(2n)!}}
\label{eq:def_hank3e}
\end{equation}
and
\begin{equation}
 \mathcal{H}_{m,n}^{(3\text{o})}:=\mathcal{H}_{2m+1,2n+1}^{(3)}=(-1)^{m+n}\frac{\Gamma(m+n+3/2)}{\sqrt{(2m+1)!(2n+1)!}},
\label{eq:def_hank3o}
\end{equation}
where $m,n\in\N_{0}$.

\begin{cor}\label{cor:hank3eo}
  Matrices $\mathcal{H}^{(3\text{e})}$ and $\mathcal{H}^{(3\text{o})}$ given by~\eqref{eq:def_hank3e} and~\eqref{eq:def_hank3o} determine 
  unique self-adjoint operators $H^{(3e)}$ and $H^{(3o)}$ on $\ell^{2}(\N_{0})$ which are unitarily equivalent via mappings $U_{e}$ and $U_{o}$ to the multiplication operator by function~$h$ on $L^{2}((0,\infty),\dd\mu)$, respectively. The function $h$ as well as the measure $\mu$ are the same as in Theorem~\ref{thm:hermite}, while the unitary operators are given as
  \[
   U_{e}:\ell^{2}(\N_{0})\to L^{2}((0,\infty),\dd\mu): e_{n}\mapsto\sqrt{2}\hat{P}_{2n}
  \]
  and
  \[
   U_{o}:\ell^{2}(\N_{0})\to L^{2}((0,\infty),\dd\mu): e_{n}\mapsto\sqrt{2}\hat{P}_{2n+1},
  \]
  where $\hat{P}_{n}$ is again the same as in Theorem~\ref{thm:hermite}.
\end{cor}

The Askey scheme of hypergeometric orthogonal polynomials contains also families of polynomials whose measures of orthogonality are supported on finite sets which means that
the associated Jacobi matrices are finite. For two of these families -- Krawtchouk and dual Hahn polynomials -- the Jacobi parameters $b_{n}$ and $a_{n}^{2}$ are polynomials
in $n$. We found only rank-one weighted Hankel matrices (or their linear combinations) in the commutant of the Jacobi matrix associated with Krawtchouk polynomials. On the other
hand, the commutant of the Jacobi matrix corresponding to the dual Hahn polynomials contains interesting finite weighted Hankel matrices. Applying the commutator method, 
we diagonalize one of them, namely, the $(N+1)\times(N+1)$ matrix $H^{(4)}$ with entries
\begin{equation}
 H_{m,n}^{(4)}=\frac{(1+\gamma)_{m+n}(1+\delta)_{2N-m-n}}
 {\sqrt{m!n!(N-m)!(N-n)!(1+\gamma)_{m}(1+\gamma)_{n}(1+\delta)_{N-m}(1+\delta)_{N-n}}},
\label{eq:def_hank4}
\end{equation}
where $m,n\in\{0,1,\dots,N\}$, $N\in\N_{0}$, and  $\gamma,\delta>-1$. Here as well as further in the text,
$(a)_{n}:=a(a+1)\dots(a+n-1)$ stands for the Pochhammer symbol. Furthermore, we denote by $\delta_{x}$ the unit 
mass Dirac delta measure supported at a point $x$.

For the purpose of a definition of the diagonalizing unitary mapping $U$, we recall the dual Hahn polynomials~\cite[Eq.~9.6.1]{koekoek-etal_10}
\begin{equation}
 R_{n}\left(\lambda(x);\gamma,\delta,N\right)=\pFq{3}{2}{-n,-\!x,\;x+\gamma+\delta+1}{\gamma+1,\,-N}{1}, \quad n\in\{0,1,\dots,N\},
\label{eq:def_dualhahn}
\end{equation}
where $N\in\N_{0}$, $\gamma,\delta>-1$, and 
\begin{equation}
\lambda(x)=x(x+\gamma+\delta+1).
\label{eq:def_lam_dualhahn}
\end{equation}

\begin{thm}\label{thm:dual_hahn}
 Let $N\in\N_{0}$ and $\gamma,\delta>-1$. The matrix $H^{(4)}$ defined by~\eqref{eq:def_hank4} is unitarily equivalent via $U:\C^{N+1}\to L^{2}(\R,\dd\mu):e_{n}\mapsto\hat{P}_{n}$, $\forall n\in\{0,1,\dots,N\}$, to the operator of multiplication by $h$ on $L^{2}(\R,\dd\mu)$, 
 where $\mu$ is the finitely supported measure
 \[
  \mu=(1+\delta)_{N}N!\sum_{x=0}^{N}\frac{(2x+\gamma+\delta+1)(1+\gamma)_{x}}{(1+x+\gamma+\delta)_{N+1}(1+\delta)_{x}(N-x)!x!}\delta_{x},
 \]
 \[
  h(x)=\binom{2N+\gamma+\delta+1}{N-x},
 \]
 and
 \[
 \hat{P}_{n}(x)=\sqrt{\frac{N!(1+\gamma)_{n}(1+\delta)_{N-n}}{n!(N-n)!(1+\delta)_{N}}}\,R_{n}\left(\lambda(x);\gamma,\delta,N\right)\!,
 \]
 for $x\in\{0,1,\dots,N\}$. Consequently, one has
 \[
  \spec\left(H^{(4)}\right)=\left\{\binom{2N+\gamma+\delta+1}{x}\,\bigg|\,x\in\{0,1,\dots,N\}\right\}\!.
 \]
\end{thm}

As a consequence, we can evaluate the following Hankel determinant.

\begin{cor}\label{cor:dual_hahn}
 The determinant of the Hankel matrix $G$ with entries
 \[
  G_{m,n}=(1+\gamma)_{m+n}(1+\delta)_{2N-m-n}, \quad m,n\in\{0,1,\dots,N\},
 \]
 reads
 \[
  \det G=\prod_{s=0}^{N}s!(1+\gamma)_{s}(1+\delta)_{s}(2N-s+2+\gamma+\delta)_{s},
 \]
 for all $\gamma,\delta\in\C$.
\end{cor}

\section{Preliminaries}\label{sec:prelim}

The commutator method consists of two fundamental steps. First, to a given Jacobi matrix $\mathcal{J}$ we find a weighted Hankel matrix $\mathcal{H}$ commuting formally 
with~$\mathcal{J}$. The formal commutation means that $\mathcal{H}\mathcal{J}$ coincides with 
$\mathcal{J}\mathcal{H}$ as elements of the space of semi-infinite matrices~$\C^{\infty,\infty}$, where the product
is to be understood as the usual matrix product. Since the Jacobi matrix is banded the products $\mathcal{H}\mathcal{J}$
and $\mathcal{J}\mathcal{H}$ are well-defined and $\mathcal{H}\mathcal{J}=\mathcal{J}\mathcal{H}$ if and only if
 \begin{align}
   &(b_{m}-b_{n})w_{m}w_{n}h_{m+n}+(a_{m-1}w_{m-1}w_{n}-a_{n-1}w_{m}w_{n-1})h_{m+n-1}\nonumber\\
 &+(a_{m}w_{m+1}w_{n}-a_{n}w_{m}w_{n+1})h_{m+n+1}=0, \label{eq:comm_rel_HJ=JH}
 \end{align}
for all $m,n\in\N_{0}$, where one has to set $a_{-1}:=0$. It will be useful to rewrite~\eqref{eq:comm_rel_HJ=JH} in a slightly different form.

\begin{lem}
 The weighted Hankel matrix $\mathcal{H}$ with the entries given by~\eqref{eq:def_weighted_hankel} 
 belongs to the commutant of $\mathcal{J}$, if and only if
 \begin{equation}
 \left(b_{m}-b_{n}\right)h_{m+n}+
 \left(\frac{a_{m-1}^{2}}{\kappa_{m-1}}-\frac{a_{n-1}^{2}}{\kappa_{n-1}}\right)h_{m+n-1}+
 \left(\kappa_{m}-\kappa_{n}\right)h_{m+n+1}=0,
\label{eq:commut_eq_reform}
\end{equation}
where
\begin{equation}
 \kappa_{n}=\frac{a_{n}w_{n+1}}{w_{n}},
\label{eq:def_kappa}
\end{equation}
holds for all $0\leq m<n$. For $m=0$, the term $a_{-1}^{2}/\kappa_{-1}$, that appears in~\eqref{eq:commut_eq_reform}, 
is to be understood as zero.
\end{lem}

\begin{rem}
 On the other hand, from~\eqref{eq:def_kappa} it follows that
\begin{equation}
  w_{n}=w_{0}\prod_{j=0}^{n-1}\frac{\kappa_{j}}{a_{j}}, \quad n\in\N_{0}.
 \label{eq:w_kappa}
\end{equation}
\end{rem}

In all cases under investigation, the Jacobi matrix $\mathcal{J}$ determines a unique self-adjoint operator 
$J$ on $\ell^{2}(\N_{0})$. We will prove a theorem that gives us a condition guaranteeing that the weighted Hankel 
matrix~$\mathcal{H}$ from the commutant of~$\mathcal{J}$ also determines a unique self-adjoint operator 
$H$ on $\ell^{2}(\N_{0})$ and, in this case, $H=h(J)$ for some Borel function $h$. This is a nontrivial
problem especially when the matrix~$\mathcal{H}$ does not give rise to a bounded operator on~$\ell^{2}(\N_{0})$
in which case more self-adjoint operators can be associated with~$\mathcal{H}$ in principle.

The second step of the commutator method comprises a determination of the function~$h$. This is 
carried out by using special properties of respective particular families of orthogonal polynomials.

\subsection{The commutant of a Jacobi matrix}

Recall that a Jacobi matrix $\mathcal{J}$ determines a unique family of orthonormal polynomials $\{\hat{P}_{n}\}_{n=0}^{\infty}$
by the three-term recurrence~\cite{akhiezer90}
\begin{equation}
 a_{n-1}\hat{P}_{n-1}(x)+(b_{n}-x)\hat{P}_{n}(x)+a_{n}\hat{P}_{n+1}(x)=0, \quad n\in\N_{0},
\label{eq:three-term_recur}
\end{equation}
with initial conditions $\hat{P}_{-1}(x):=0$ and $\hat{P}_{0}(x):=1$.

First, we prove an algebraic result that gives a description of the commutant of a Jacobi matrix.
It turns out that any matrix commuting with a Jacobi matrix has to be symmetric. We will use just this consequence
in order to slightly simplify assumptions of a forthcoming theorem but the complete result can be of independent interest.

\begin{thm}\label{thm:commut_matr_Jac}
 Let $\mathcal{J}$ be a Jacobi matrix, $\{\hat{P}_{n}\}_{n=0}^{\infty}$ the family of orthonormal polynomials determined by~$\mathcal{J}$, and $\mathcal{A}\in\C^{\infty,\infty}$. Then
 \begin{equation}
 \mathcal{J}\mathcal{A}=\mathcal{A}\mathcal{J} \quad \Longleftrightarrow \quad
 \exists \{\alpha_{k}\}_{k=0}^{\infty}\subset\C,\quad 
 \mathcal{A}=\sum_{k=0}^{\infty}\alpha_{k}\hat{P}_{k}(\mathcal{J}).
 \label{eq:commut_equiv_Jac}
 \end{equation}
 As a consequence, if $\mathcal{J}\mathcal{A}=\mathcal{A}\mathcal{J}$, then $\mathcal{A}=\mathcal{A}^{T}$.
\end{thm}

\begin{rem}
 The sum on the right-hand side of the equivalence~(\ref{eq:commut_equiv_Jac}) is actually finite for any matrix element. More precisely, one has
 \[
  \mathcal{A}_{m,n}=\sum_{k=|m-n|}^{m+n}\alpha_{k}\left(\hat{P}_{k}(\mathcal{J})\right)_{m,n}\!, \quad \forall m,n\in\N_{0}.
 \]
 This follows from the well-known relation $\hat{P}_{n}(\mathcal{J})e_{0}=e_{n}$ valid for all $n\in\N_{0}$ that 
 implies
 \begin{align*}
  \left(\hat{P}_{k}(\mathcal{J})\right)_{m,n}&=\langle e_{m}, \hat{P}_{k}(\mathcal{J})e_{n}\rangle_{\ell^{2}(\N_{0})}=
  \langle e_{0}, \hat{P}_{m}(\mathcal{J})\hat{P}_{k}(\mathcal{J})\hat{P}_{n}(\mathcal{J})e_{0}\rangle_{\ell^{2}(\N_{0})}\\
  &=\int_{\R}\hat{P}_{k}(x)\hat{P}_{m}(x)\hat{P}_{n}(x)\dd\mu(x),
 \end{align*}
 where $\mu$ is a (not necessarily unique) measure of orthogonality of $\{\hat{P}_{n}\}_{n=0}^{\infty}$. 
 By the orthogonality of~$\{\hat{P}_{n}\}_{n=0}^{\infty}$, the above integral vanishes whenever 
 $k+m<n$ or $m+n<k$ or $k+n<m$. Hence the range for the index $k$ can be restricted to $|m-n|\leq k\leq m+n$.
\end{rem}

\begin{proof}
 Clearly, $\mathcal{J}$ and $\hat{P}_{n}(\mathcal{J})$ commute for any $n\in\N_{0}$. Hence the proof of the implication $(\Longleftarrow)$ is immediate.
 Also, if the equivalence is proved, then the consequence about the symmetry of $\mathcal{A}$ follows readily from the 
 fact that $\hat{P}_{n}(\mathcal{J})$ is symmetric for all $n\in\N_{0}$.
 
 It remains to prove the implication $(\Longrightarrow)$. First, suppose $\mathcal{J}\mathcal{A}=\mathcal{A}\mathcal{J}$
 and that the first column of $\mathcal{A}$ vanishes, i.e., $\mathcal{A}e_{0}=0$. We shall show by mathematical induction 
 in $n$ that $\mathcal{A}e_{n}=0$ for all $n\in\N_{0}$, meaning that $\mathcal{A}=0$.
 
 Let $n\in\N_{0}$. Assuming $\mathcal{A}e_{k}=0$ for all $k\in\{0,1,\dots,n\}$, we have $\mathcal{A}\mathcal{J}e_{k}=0$ for all $k\in\{0,1,\dots,n\}$.
 But
 \[
  \mathcal{J}e_{0}\in\spn\{e_{0},e_{1}\} \quad \mbox{ and } \quad \mathcal{J}e_{k}\in\spn\{e_{k-1},e_{k},e_{k+1}\}, \; \mbox{ for } k\in\N,
 \]
 and thus we deduce that $\mathcal{J}_{n+1,n}\mathcal{A}e_{n+1}=0$. By our assumptions $\mathcal{J}_{n+1,n}\neq0$
 and therefore $\mathcal{A}e_{n+1}=0$.
 
 Second, consider the general case when $\mathcal{J}\mathcal{A}=\mathcal{A}\mathcal{J}$ and put
 \[
  \tilde{\mathcal{A}}:=\sum_{k=0}^{\infty}\mathcal{A}_{k,0}\hat{P}_{k}(\mathcal{J}).
 \]
 Since 
 \[
  \left(\hat{P}_{k}(\mathcal{J})\right)_{m,0}=\delta_{k,m}, \quad \forall k,m\in\N_{0},
 \]
 the first columns of $\mathcal{A}$ and $\tilde{\mathcal{A}}$ coincide. By applying the first argument to 
 $\mathcal{A}-\tilde{\mathcal{A}}$, we conclude that $\mathcal{A}=\tilde{\mathcal{A}}$.
\end{proof}

\subsection{On the definition of an operator associated with a weighted Hankel matrix}

It is an old question whether we can define a closed and densely defined linear operator on 
$\ell^{2}(\N_{0})$ whose matrix representation with respect to the standard basis $\{e_{n} \mid n\in\N_{0}\}$ 
coincides with a given semi-infinite matrix~$\mathcal{A}$.
This is possible when the rows and columns of $\mathcal{A}$ can be identified with elements of $\ell^{2}(\N_{0})$.
Below we briefly summarize the standard procedure showing how to prescribe the operators to $\mathcal{A}$.
We restrict ourself to real symmetric matrices. This slightly simplifies the procedure and is sufficient 
for our needs; for a general summary see, for example,~\cite[Sec.~2.1]{beckermann_jcam01}.

Denote 
\[
\ell_{0}:=\spn\{e_{n} \mid n\in\N_{0}\}. 
\]
Since columns of $\mathcal{A}$ can be identified with vectors in $\ell^{2}(\N_{0})$, we can define
an auxiliary operator $\dot{A}$ acting on vectors 
\[
x\in\Dom(\dot{A}):=\ell_{0}
\]
by the formal matrix product $\dot{A}x:=\mathcal{A}x$, where $x$ is understood as an infinite column vector.
Assuming $\mathcal{A}=\mathcal{A}^{T}\in\R^{\infty,\infty}$, $\dot{A}$ is symmetric and hence
closable in~$\ell^{2}(\N_{0})$. Then we may define $A_{\min}$ as the closure of $\dot{A}$.
While $A_{\min}$ is the minimal closed operator associated with $\mathcal{A}$ such that
$\ell_{0}\subset\Dom(A_{\min})$, the maximal domain operator $A_{\max}$ acts again by the matrix multiplication
by $\mathcal{A}$ but it is equipped with the domain
\[
 \Dom(A_{\max}):=\{x\in\ell^{2}(\N_{0}) \mid \mathcal{A}x \in\ell^{2}(\N_{0})\}.
\]
It is an easy exercise to show that
\[
 (A_{\min})^{*}=A_{\max} \quad\mbox{ and }\quad (A_{\max})^{*}=A_{\min}.
\]
Consequently, $A_{\max}$ is closed.

Clearly, $A_{\min}\subset A_{\max}$ but the equality does not hold in general.
If $A_{\min}=A_{\max}$, then the matrix $\mathcal{A}$ determines a unique
self-adjoint operator $A:=A_{\min}=A_{\max}$ with $\ell_{0}\subset\Dom(A)$.
In this case, we say that the matrix $\mathcal{A}$ is \emph{proper}.

In general, it is a difficult task to decide whether a given matrix~$\mathcal{A}$ is proper.
In the particular case when $\mathcal{A}=\mathcal{J}$ is a Jacobi matrix, the properness
is very well understood. One can encounter various terminology in literature equivalent to
the properness of $\mathcal{J}$, for example, $\mathcal{J}$ is said to be in the limit-point case or
$\dot{J}$ essentially self-adjoint. In addition, the corresponding Hamburger moment problem is 
in the determinate case if and only if $\mathcal{J}$ is proper, see~\cite{akhiezer90}.

Assume $\mathcal{J}$ to be a proper Jacobi matrix and $J$ the corresponding self-adjoint Jacobi operator.
Then there is a unique probability measure $\mu$ on $\R$ with finite moments such that $J$ is unitarily equivalent to
the multiplication operator by the coordinate function on $L^{2}(\R,\dd\mu)$. The unitary operator $U$
satisfies
\[
 U:\ell^{2}(\N_{0})\to L^{2}(\R,\dd\mu): e_{n} \mapsto \hat{P}_{n}, \quad \forall n\in\N_{0},
\]
where $\{\hat{P}_{n}\}_{n=0}^{\infty}$ is the sequence of polynomials orthonormal with respect to $\mu$ and 
determined by $\mathcal{J}$, i.e., by the three-term recurrence relation~\eqref{eq:three-term_recur} with the standard initial conditions.

The next theorem gives a necessary and sufficient condition for a real matrix commuting formally 
with a proper Jacobi matrix to be proper as well. We will use it to guarantee that the weighted Hankel 
matrices listed in Section~\ref{sec:main} determine unique self-adjoint operators in~$\ell^{2}(\N_{0})$.

\begin{thm}\label{thm:def_H}
 Let $\mathcal{J}$ be a proper Jacobi matrix, $\mathcal{H}\in\R^{\infty,\infty}$ with columns in $\ell^{2}(\N_{0})$, and
 $\mathcal{J}\mathcal{H}=\mathcal{H}\mathcal{J}$. Then $\mathcal{H}$ is proper if and only if
 the set of all polynomials $\C[x]$ is dense in $L^{2}(\R,(|h|+1)^{2}\dd\mu)$, where $h:=U\mathcal{H}e_{0}$
 and $\mu$ and $U$ are determined by $J$ as above. Moreover, in this case, $H$ is unitarily equivalent via $U$ 
 to the multiplication operator by the function $h$ on $L^{2}(\R,\dd\mu)$, i.e.,
 \begin{equation}
  H=h\!\left(J\right).
 \label{eq:H_multiplic_op_h}
 \end{equation}
\end{thm}

\begin{proof}
 First, note that $\mathcal{H}=\mathcal{H}^{T}$ according to Theorem~\ref{thm:commut_matr_Jac} and hence columns as well as rows of $\mathcal{H}$ are in~$\ell^{2}(\N_{0})$. 
 Therefore $\dot{H}$ is a well defined and closable operator. Put
 \begin{equation}
  \dot{H}_{U}:=U\dot{H}U^{-1}.
 \label{eq:H_U_dot_unit_trans}
 \end{equation}
 Then $\Dom\dot{H}_{U}=\C[x]$ because $U$ is an isomorphism between $\ell_{0}$ and $\C[x]$. First, we show
 that $\dot{H}_{U}$ acts as the multiplication operator by the function $h=\dot{H}_{U}1$.
 
 For $J_{U}:=UJU^{-1}$, one has
 \[
  \Dom(J_{U})=\{\psi\in L^{2}(\R,\dd\mu) \mid x\psi(x)\in L^{2}(\R,\dd\mu)\} \;\mbox{ and }\; (J_{U}\psi)(x)=x\psi(x).
 \]
 Since $\mu$ has all moments finite, $\C[x]\subset\Dom J_{U}$ and we may define $\dot{J}_{U}:=J_{U}\upharpoonleft\C[x]$.
 Then the formal commutation of $\mathcal{J}$ and $\mathcal{H}$ means that
 \[
  (\dot{H}_{U}\dot{J}_{U}\psi)(x)=(\dot{J}_{U}\dot{H}_{U}\psi)(x), \quad \forall\psi\in\C[x].
 \]
 Hence $\Ran\dot{H}_{U}\subset\Dom J_{U}$ and the commutation relation can be rewritten as
 $\dot{H}_{U}\dot{J}_{U}=J_{U}\dot{H}_{U}$ which implies that
 \[
  \dot{H}_{U}(x\psi(x))=x(\dot{H}_{U}\psi)(x), \quad \forall\psi\in\C[x].
 \]
 Since $h=\dot{H}_{U}1$, one can use the above relation to verify that
 \[
  \dot{H}_{U}(x^{n})=h(x)x^{n}, \quad \forall n\in\N_{0},
 \]
 by mathematical induction in $n$. Thus, $\dot{H}_{U}\psi=h\psi$ for all $\psi\in\C[x]$.
 
 We known that $\mathcal{H}$ is proper if and only if $\dot{H}$ is  essentially self-adjoint.
 In view of~\eqref{eq:H_U_dot_unit_trans}, this is equivalent to the essential self-adjointness
 of $\dot{H}_{U}$. We have to show that this happens if and only if $\C[x]$ is dense in $L^{2}(\R,(|h|+1)^{2}\dd\mu)$.

 Let $H_{U}:=h(J_{U})$, i.e.,
 \[
  \Dom(H_{U})=\{\psi\in L^{2}(\R,\dd\mu) \mid h(x)\psi(x)\in L^{2}(\R,\dd\mu)\} \;\mbox{ and }\; (H_{U}\psi)(x)=h(x)\psi(x).
 \]
 Since $h$ is real-valued, $H_{U}$ is self-adjoint. Clearly $\dot{H}_{U}\subset H_{U}$ and hence 
 $\dot{H}_{U}$ is essentially self-adjoint if and only if $\overline{\dot{H}_{U}}=H_{U}$.
 
 Note that $\C[x]$ is dense in $L^{2}(\mathbb{R},(|h|+1)^{2}\dd\mu)$
 if and only if $(|H_{U}|+1)(\C[x])$ is dense in $L^{2}(\mathbb{R},\dd\mu)$.
 We have $\Dom H_{U}=\Dom|H_{U}|$ and $\|H_{U}\psi\|=\|\,|H_{U}|\psi\|$ for all $\psi\in\Dom H_{U}$.
 Moreover, $\Ran(|H_{U}|+1)^{\perp}=\Ker(|H_{U}|+1)=\{0\}.$

 The equation $\overline{\dot{H}_{U}}=H_{U}$ means exactly that $\C[x]$ is a core of $H_{U}$ and this happens
 if and only if $\C[x]$ is a core of $|H_{U}|$. Clearly, for a
 sequence $\{\psi_{n}\}$ from $\Dom H_{U}$ and $\psi\in\Dom H_{U}$, it holds true that
 \[
 \lim_{n\to\infty}(|H_{U}|+1)\psi_{n}=(|H_{U}|+1)\psi \;\ \mbox{ in } L^{2}(\R,\dd\mu)
 \]
 if and only if
 \[
 \lim_{n\to\infty}\psi_{n}=\psi \;\mbox{ and }\; \lim_{n\to\infty}|H_{U}|\psi_{n}=|H_{U}|\psi \;\ \mbox{ in } L^{2}(\R,\dd\mu).
 \]
  It follows that $\C[x]$ is a core of $|H_{U}|$ if and only if
 \[
 \overline{(|H_{U}|+1)(\C[x])}\supset\Ran(|H_{U}|+1).
 \]
 But $\Ran(|H_{U}|+1)$ is dense in $L^{2}(\R,\dd\mu)$ because $\Ker(|H_{U}|+1)=\{0\}$. Therefore $\C[x]$ is a core of $|H_{U}|$
 if and only if $(|H_{U}|+1)(\C[x])$ is dense in $L^{2}(\mathbb{R},\dd\mu)$. This proves the desired
 equivalence and also shows that
 \[
  H_{U}=h\!\left(J_{U}\right)
 \]
 from which the formula~\eqref{eq:H_multiplic_op_h} follows.
\end{proof}

The classical problem how to characterize measures $\nu$ on $\R$ with finite moments such that
$\C[x]$ is dense in $L^{2}(\R,\dd\nu)$ was solved by M.~Riesz. The theorem says that $\C[x]$ is dense in $L^{2}(\R,\dd\nu)$
if and only if $\nu$ is the so called N-extremal solution of the respective Hamburger moment problem, 
see~\cite[Thm.~2.3.3]{akhiezer90}. It need not be easy to decide whether a given measure~$\nu$ is N-extremal.
However, in the particular case when the corresponding moment problem is determinate, $\nu$ is N-extremal. 
There exist various sufficient conditions for the Hamburger moment problem to be determinate. For instance, 
if there exists $\epsilon>0$ such that
\[
\int_{\mathbb{R}}e^{\epsilon|x|}\dd\nu(x)<\infty,
\]
then the Hamburger moment problem corresponding to~$\nu$ is in the determinate case, 
see~\cite[Prob.~12, p.~86]{akhiezer90} or~\cite[Cor.~4.11]{schmudgen17}. By combining the latter
condition with Theorem~\ref{thm:def_H}, the following statement immediately follows.

\begin{cor}
 Let $\mathcal{J}$ be a proper Jacobi matrix, $\mathcal{H}\in\R^{\infty,\infty}$ with columns in $\ell^{2}(\N_{0})$, and
 $\mathcal{J}\mathcal{H}=\mathcal{H}\mathcal{J}$. If there exists $\epsilon>0$ such that
 \begin{equation}
 \int_{\R}e^{\epsilon|x|}(|h(x)|+1)^{2}\dd\mu(x)<\infty,
 \label{eq:int_cond_H_proper}
 \end{equation}
 then $\mathcal{H}$ is proper and one has
 \[
  H=h(J).
 \]
 Here $\mu$, $U$, and $h$ are as in Theorem~\ref{thm:def_H}.
\end{cor}

\section{Proofs of the main results}\label{sec:proofs}

In this section, the theorems of Section~\ref{sec:main} are proven. Note that every Jacobi matrix that 
appears in the subsections below is proper. It can be verified, for example, by using the Carleman condition:
\[
 \sum_{n=0}^{\infty}\frac{1}{|a_{n}|}=\infty
\]
in each of the cases; see~\cite[Prob.~1, p.~24]{akhiezer90}.
In fact, each of the Jacobi matrices corresponds to a family of hypergeometric orthogonal polynomials from the Askey scheme 
and it is well known that no Hamburger moment problem associated with a family of hypergeometric orthogonal polynomials from 
the Askey scheme is in the indeterminate case~\cite{ismail05}. This is not the case for the polynomials from the $q$-Askey scheme, however.

\subsection{Proof of Theorem~\ref{thm:meixpol_lag_meix}}

First, one immediately checks that the columns of $\mathcal{H}^{(1)}$ defined by~\eqref{eq:def_hank1}
belong to $\ell^{2}(\N_{0})$. Indeed, for $n\in\N_{0}$ fixed, the Stirling formula yields
\[
 \mathcal{H}^{(1)}_{m,n}=\frac{k^{n+m}m^{n+(\alpha-1)/2}}{\sqrt{n!\,\Gamma(n+\alpha)}}\left(1+o(1)\right), \quad \mbox{ as }\; m\to\infty.
\]
Thus, for $k\in(0,1)$ and $\alpha>0$, every column of $\mathcal{H}^{(1)}$ represents a square summable sequence.

Put
\begin{equation}
 b_{n}:=n \quad \mbox{ and } \quad a_{n}:=-k\sqrt{(n+1)(n+\alpha)},
\label{eq:def_a_b_hankel1}
\end{equation}
for $n\in\N_{0}$. By substituting from~\eqref{eq:def_a_b_hankel1} and taking $\kappa_{n}:=-k$ 
in~\eqref{eq:commut_eq_reform}, one gets the difference equation
\[
 (m-n)h_{m+n}-k(m-n)(m+n-1+\alpha)h_{m+n-1}=0.
\]
After canceling the term $(m-n)$ in the above equation, we arrive at a difference equation entirely in $m+n$.
Hence, if we denote $\ell:=m+n$, we obtain the first order ordinary difference equation
\[
 h_{\ell}-k(\ell-1+\alpha)h_{\ell-1}=0, \quad \ell\in\N.
\]
Its solution reads
\begin{equation}
 h_{\ell}=k^{\ell}(\alpha)_{\ell}\,h_{0}=k^{\ell}\Gamma(\ell+\alpha), \quad \ell\in\N_{0},
\label{eq:h_n_hank1_meixpol}
\end{equation}
where we put $h_{0}:=\Gamma(\alpha)$. The weights can be computed readily from~\eqref{eq:w_kappa} getting
\begin{equation}
 w_{n}=\frac{w_{0}}{\sqrt{n!\,(\alpha)_{n}}}=\frac{1}{\sqrt{n!\,\Gamma(n+\alpha)}}, \quad n\in\N_{0},
\label{eq:w_n_hank1_meixpol}
\end{equation}
where we chose $w_{0}:=1/\sqrt{\Gamma(\alpha)}$.
In summary, the weighted Hankel matrix with elements $w_{m}w_{n}h_{m+n}$, $m,n\in\N_{0}$, coincides with $\mathcal{H}^{(1)}$ from~\eqref{eq:def_hank1}
and it commutes with the Jacobi matrix~$\mathcal{J}$ determined by~\eqref{eq:def_a_b_hankel1} for any $k\in(0,1)$ and $\alpha>0$.

\subsubsection{The case (i)}
Recall the Meixner--Pollaczek polynomials~\eqref{eq:def_meix-polllac} and put 
\begin{equation}
\hat{P}_{n}(x):=\left(\frac{\Gamma(2\lambda)\,n!}{\Gamma(n+2\lambda)}\right)^{\!1/2}P_{n}^{(\lambda)}\!\left(x\cot\phi\,;\phi\right)\!, \quad n\in\N_{0},
\label{eq:def_hatP_hank1_meixpol}
\end{equation}
and $\hat{P}_{-1}(x):=0$. The three-term recurrence relation for the Meixner--Pollaczek polynomials~\cite[Eq.~9.7.3]{koekoek-etal_10} yields
\[
 \sqrt{(n+1)(n+2\lambda)}\hat{P}_{n+1}(x)-2\cos(\phi)(n+\lambda+x)\hat{P}_{n}(x)+\sqrt{n(n-1+2\lambda)}\hat{P}_{n-1}(x)=0,
\]
for $n\in\N_{0}$. Hence, denoting by~$\mathcal{I}$ the identity matrix, the vector $(\hat{P}_{0}(x),\hat{P}_{1}(x),\hat{P}_{2}(x),\dots)$
is a formal eigenvector of the Jacobi matrix $-\mathcal{J}-\lambda\mathcal{I}$ corresponding to the eigenvalue $x$, 
if we identify the parameters according to the equalities
\begin{equation}
 \alpha=2\lambda \quad \mbox{ and } \quad k=\frac{1}{2\cos\phi}
\label{eq:param_iden_meixpol}
\end{equation}
and restrict $\phi\in(0,\pi/3)$.

By Carleman's condition, $\mathcal{J}$ determines a unique Jacobi operator $J$ and
$\{\hat{P}_{n} \mid n\in\N_{0}\}$ is an orthonormal basis of $L^{2}(\R,\dd\mu)$
with
\begin{equation}
\frac{\dd\mu}{\dd x}(x)=\frac{\cos\phi}{\pi\Gamma(2\lambda)}\left(2\sin\phi\right)^{2\lambda-1}
e^{(2\phi-\pi)x\cot(\phi)}\!\left|\Gamma\left(\lambda+\ii x\cot(\phi)\right)\right|^{2},
\label{eq:mu_hank1_meixpol}
\end{equation}
for $x\in\R$. The above measure was deduced from the measure of orthogonality for the Meixner--Pollaczek polynomials~\cite[Eq.~9.7.2]{koekoek-etal_10}.
By applying~\eqref{eq:param_iden_meixpol} and the identity $2\arccos(x)=\pi-2\arcsin(x)$ in~\eqref{eq:mu_hank1_meixpol},
we deduce the formula for $\mu$ in terms of the parameters~$\alpha$ and~$k$ from the case~(i) of Theorem~\ref{thm:meixpol_lag_meix}. Similarly, if $\hat{P}_{n}$ defined by \eqref{eq:def_hatP_hank1_meixpol} is expressed in terms of~$\alpha$ and~$k$, we get the respective formula from Theorem~\ref{thm:meixpol_lag_meix} in case~(i).

Next, we compute the function $h$ according to Theorem~\ref{thm:def_H} applied to the matrices $\mathcal{H}^{(1)}$ and $-\mathcal{J}-\lambda\mathcal{I}$.
Clearly, since $\mathcal{H}^{(1)}$ and $\mathcal{J}$ commute, $\mathcal{H}^{(1)}$ commutes with $-\mathcal{J}-\lambda\mathcal{I}$, too.
By the formula from Theorem~\ref{thm:def_H}, one has
\[
h(x)=\sum_{n=0}^{\infty}\mathcal{H}_{n,0}^{(1)}\hat{P}_{n}(x)=w_{0}\sum_{n=0}^{\infty}w_{n}h_{n}\hat{P}_{n}(x).
\]
By substituting from~\eqref{eq:h_n_hank1_meixpol}, \eqref{eq:w_n_hank1_meixpol}, and \eqref{eq:def_hatP_hank1_meixpol}
in the equation above and taking also~\eqref{eq:param_iden_meixpol} into account, we obtain
\begin{equation}
h(x)=\sum_{n=0}^{\infty}(2\cos(\phi))^{-n}\, P_{n}^{(\lambda)}\!\left(x\cot(\phi);\phi\!\right)
 = \left(2\cos(\phi)\right)^{2\lambda}e^{2x\phi\cot(\phi)}.
\label{eq:h_hank1_meixpol}
\end{equation}
The last equality holds due to the generating function formula~\cite[Eq.~9.7.11]{koekoek-etal_10}
\begin{equation}
\sum_{n=0}^{\infty}P_{n}^{(\lambda)}(x;\phi)\, t^{n}=
\left(1-e^{\ii\phi}t\right)^{-\lambda+\ii x}\left(1-e^{-\ii\phi}t\right)^{-\lambda-\ii x}.
\label{eq:gener_func_meix-pol}
\end{equation}
By using~\eqref{eq:param_iden_meixpol} again, we can express the right-hand side of~\eqref{eq:h_hank1_meixpol} 
in terms of the original parameters $\alpha$ and $k$ which yields the formula for $h$ from the case~(i).

Finally, to conclude that $\mathcal{H}^{(1)}$ determines a unique self-adjoint operator $H^{(1)}$ and $H^{(1)}=h(J)$, it suffices to
verify the condition~\eqref{eq:int_cond_H_proper}. We make use of the asymptotic expansion~\cite[Eq.~5.11.9]{dlmf}
\begin{equation}
 |\Gamma(x+\ii y)|^{2}=2\pi|y|^{2x-1}e^{-\pi|y|}\left(1+o(1)\right), \quad y\to\pm\infty,
\label{eq:gamma_asympt_imag}
\end{equation}
where $x\in\R$. Then, recalling formulas~\eqref{eq:mu_hank1_meixpol} and~\eqref{eq:h_hank1_meixpol},
one easily shows that there is a constant $C_{\lambda,\phi}>0$ not depending on $x$ such that
\[
 (|h(x)|+1)^{2}\,\frac{\dd\mu}{\dd x}(x)\leq C_{\lambda,\phi}|x|^{2\lambda-1}e^{2|x|(3\phi-\pi)\cot(\phi)},
\]
for $|x|$ sufficiently large. Since $0<\phi<\pi/3$, one can find $\epsilon>0$ small enough such that the condition~\eqref{eq:int_cond_H_proper} is fulfilled.

In total, the unitary mapping
\[
 U:\ell^{2}(\N_{0})\to L^{2}(\R,\dd\mu): e_{n}\mapsto\hat{P}_{n}
\]
transforms $H^{(1)}$ to the multiplication operator by the function $h$ on $L^{2}(\R,\dd\mu)$. As an immediate 
consequence of formulas~\eqref{eq:mu_hank1_meixpol} and~\eqref{eq:h_hank1_meixpol}, one shows that
\[
 \spec\left(H^{(1)}\right)=\spec_{\text{ac}}\left(H^{(1)}\right)=\overline{h(\R)}=[\,0,+\infty).
\]

\subsubsection{The case (ii)}
We define in terms of the Laguerre polynomials~\eqref{eq:def_laguerre}
\begin{equation}
\hat{P}_{n}(x):=\left(\frac{\Gamma(\alpha)\,n!}{\Gamma(n+\alpha)}\right)^{\!1/2}L^{(\alpha-1)}_{n}(2x), \quad n\in\N_{0},
\label{eq:def_hatP_hank1_laguerre}
\end{equation}
and $\hat{P}_{-1}(x):=0$. It follows from the three-term recurrence relation for the Laguerre polynomials~\cite[Eq.~9.12.3]{koekoek-etal_10} that
\[
 \sqrt{(n+1)(n+\alpha)}\hat{P}_{n+1}(x)-(2n+\alpha-2x)\hat{P}_{n}(x)+\sqrt{n(n-1+\alpha)}\hat{P}_{n-1}(x)=0,
\]
Consequently, the vector $(\hat{P}_{0}(x),\hat{P}_{1}(x),\hat{P}_{2}(x),\dots)$ is a formal eigenvector of the 
Jacobi matrix $\mathcal{J}+\frac{\alpha}{2}\mathcal{I}$ corresponding to the eigenvalue $x$, 
where the entries of $\mathcal{J}$ are given by~\eqref{eq:def_a_b_hankel1} with $k=1/2$.

The matrix $\mathcal{J}$ determines a unique Jacobi operator $J$ and $\{\hat{P}_{n} \mid n\in\N_{0}\}$ is an orthonormal basis of $L^{2}(\R,\dd\mu)$, where $\mu$
is an absolutely continuous measure supported on $[0,\infty)$ with the density
\begin{equation}
\frac{\dd\mu}{\dd x}(x)=\frac{2^{\alpha}}{\Gamma(\alpha)}x^{\alpha-1}e^{-2x}, \quad x\in(0,\infty),
\label{eq:mu_hank1_laguerre}
\end{equation}
which follows from the orthogonality relation for the Laguerre polynomials~\cite[Eq.~9.12.2]{koekoek-etal_10}.
This yields the formula for the measure~$\mu$ stated in case~(ii) of Theorem~\ref{thm:meixpol_lag_meix}
as well as~\eqref{eq:def_hatP_hank1_laguerre} coincides with the formula for $\hat{P}_{n}$ from Theorem~\ref{thm:meixpol_lag_meix} case~(ii).

According to Theorem~\ref{thm:def_H} applied to the commuting matrices $\mathcal{H}^{(1)}$ and $\mathcal{J}+\frac{\alpha}{2}\mathcal{I}$,
the function $h$ reads
\[
h(x)=\sum_{n=0}^{\infty}\mathcal{H}_{n,0}^{(1)}\hat{P}_{n}(x)=w_{0}\sum_{n=0}^{\infty}w_{n}h_{n}\hat{P}_{n}(x)=\sum_{n=0}^{\infty}2^{-n}L_{n}^{(\alpha-1)}(2x),
\]
where we have used formulas~\eqref{eq:h_n_hank1_meixpol}, \eqref{eq:w_n_hank1_meixpol} with $k=1/2$, and~\eqref{eq:def_hatP_hank1_laguerre}.
By making use of the generating function formula for the Laguerre polynomials~\cite[Eq.~9.12.10]{koekoek-etal_10}
\[
\sum_{n=0}^{\infty}L_{n}^{(\alpha)}(x)t^{n}=(1-t)^{-\alpha-1}\exp\!\left(\frac{tx}{t-1}\right)\!,
\]
we obtain
\begin{equation}
 h(x)=2^{\alpha}e^{-2x}, \quad x\in(0,\infty),
\label{eq:h_hank1_laguerre}
\end{equation}
which is the formula for $h$ from case~(ii) of Theorem~\ref{thm:meixpol_lag_meix}. Since $h$ is bounded, the verification of the condition~\eqref{eq:int_cond_H_proper} with $\mu$ given by~\eqref{eq:mu_hank1_laguerre}
is immediate. As a direct consequence of formulas~\eqref{eq:mu_hank1_laguerre} and~\eqref{eq:h_hank1_laguerre}, one has
\[
\spec\left(H^{(1)}\right)=\spec_{\text{ac}}\left(H^{(1)}\right)=\overline{h((0,\infty))}=[0,2^{\alpha}]\quad\mbox{ and }\quad \|H^{(1)}\|=2^{\alpha},
\]
provided that $k=1/2$.

\subsubsection{The case (iii)}
In the case when $k\in(0,1/2)$, we put
\begin{equation}
\hat{P}_{n}(x):=c^{n/2}\left(\frac{(\beta)_{n}}{n!}\right)^{\!1/2}M_{n}\left(\frac{1+c}{1-c}x;\beta,c\right), \quad n\in\N_{0},
\label{eq:def_hatP_hank1_meixner}
\end{equation}
and $\hat{P}_{-1}(x):=0$, where $M_{n}$ denotes the Meixner polynomials~\eqref{eq:def_meixner}.

The three-term recurrence relation for the Meixner polynomials~\cite[Eq.~9.10.3]{koekoek-etal_10} implies that
\[
 \sqrt{c(n+1)(n+\beta)}\hat{P}_{n+1}(x)-\left((1+c)(n-x)+c\beta\right)\hat{P}_{n}(x)+\sqrt{cn(n-1+\beta)}\hat{P}_{n-1}(x)=0,
\]
for $n\in\N_{0}$. Thus, the vector $(\hat{P}_{0}(x),\hat{P}_{1}(x),\hat{P}_{2}(x),\dots)$ is a formal eigenvector of the Jacobi matrix $\mathcal{J}+\frac{c\beta}{1+c}\mathcal{I}$ 
corresponding to the eigenvalue $x$, where the entries of $\mathcal{J}$ are given by~\eqref{eq:def_a_b_hankel1} and the parameters are identified by equalities
\begin{equation}
 \alpha=\beta \quad \mbox{ and } \quad k=\frac{\sqrt{c}}{1+c},
\label{eq:param_iden_meixner}
\end{equation}
with $\beta>0$ and $c\in(0,1)$. In addition, $\{\hat{P}_{n} \mid n\in\N_{0}\}$ is an orthonormal basis of $L^{2}(\R,\dd\mu)$, with 
\begin{equation}
\mu=(1-c)^{\beta}\sum_{n=0}^{\infty}\frac{(\beta)_{n}}{n!}c^{n}\delta_{n(1-c)/(1+c)},
\label{eq:mu_hank1_meixner}
\end{equation}
as it follows from the orthogonality relation for the Meixner polynomials~\cite[Eq.~9.10.2]{koekoek-etal_10}.

Similarly as before, Theorem~\ref{thm:meixpol_lag_meix} together with formulas~\eqref{eq:h_n_hank1_meixpol}, \eqref{eq:w_n_hank1_meixpol}, \eqref{eq:param_iden_meixner}, 
and~\eqref{eq:def_hatP_hank1_meixner} yield
\begin{equation}
 h(x)=w_{0}\sum_{n=0}^{\infty}w_{n}h_{n}\hat{P}_{n}(x)=\sum_{n=0}^{\infty}\frac{(\beta)_{n}}{n!}\!\left(\frac{c}{1+c}\right)^{\!n}\!M_{n}\left(\frac{1+c}{1-c}x;\beta,c\right)=
 (1+c)^{\beta}c^{x(1+c)/(1-c)},
\label{eq:h_hank1_meixner}
\end{equation}
where we have used the generating function formula~\cite[Eq.~9.10.11]{koekoek-etal_10}
\[
 \sum_{n=0}^{\infty}\frac{(\beta)_{n}}{n!}M_{n}(x;\beta,c)t^{n}=\left(1-\frac{t}{c}\right)^{x}(1-t)^{-x-\beta}.
\]

The verification of the condition~\eqref{eq:int_cond_H_proper} is again straightforward here. Alternatively, one can use the general operator theory in this case
since one can readily show that the matrix~\eqref{eq:def_hank1} determines bounded and hence unique self-adjoint operator on~$\ell^{2}(\N_{0})$. To do so, one can
check, for instance, that
\[
 \sum_{m,n=0}^{\infty}\left|\mathcal{H}^{(1)}_{m,n}\right|<\infty,
\]
for $k\in(0,1/2)$, which implies that $H^{(1)}$ is actually a trace class operator. Since the Jacobi operator $J$ has simple spectrum and commutes with~$H^{(1)}$, it has to hold that $H^{(1)}=h(J)$
for a Borel function $h$. This follows from a general argument, see, for instance, \cite[Lemma~6.4]{varadarajan85} or \cite[Prop.~1.9, Suppl.~1]{berezin-shubin91}.
Then the function $h$ can be computed as in~\eqref{eq:h_hank1_meixner}.

Noticing that $0$ is not an eigenvalue of $T_{h}$, it follows from the obtained formulas~\eqref{eq:mu_hank1_meixner} and~\eqref{eq:h_hank1_meixner} that, for $k\in(0,1/2)$, $H^{(1)}$ is a trace class operator with 
\[
\spec\left(H^{(1)}\right)\setminus\{0\}=\spec_{\text{p}}\left(H^{(1)}\right)=h\left(\frac{1-c}{1+c}\N_{0}\right)=(1+c)^{\beta}c^{\N_{0}}
\]
and $\|H^{(1)}\|=(1+c)^{\beta}$. If the second relation in~\eqref{eq:param_iden_meixner} is inverted, one gets
\[
 c=\frac{1-2k^{2}-\sqrt{1-4k^{2}}}{2k^{2}}.
\]
By using the above equation to express~\eqref{eq:def_hatP_hank1_meixner}, \eqref{eq:mu_hank1_meixner}, and~\eqref{eq:h_hank1_meixner} as well as the consequences for the spectrum of $H^{(1)}$ 
in terms of the parameters $\alpha$ and $k$, one arrives at the formulas from Theorem~\ref{thm:meixpol_lag_meix} in case~(iii).

\subsection{Proof of Theorem~\ref{thm:meixpol}}

First, a straightforward application of the Stirling formula shows that
\[
\mathcal{H}_{m,n}^{(2)}=(-1)^{m(m-1)/2+n(n-1)/2}\sqrt{\frac{\Gamma(n+2\lambda)}{n!}}2^{2\lambda}m^{-\lambda-1/2}\left(1+o(1)\right), \quad \mbox{ as } m\to\infty,
\]
for $m+n$ even and $n\in\N_{0}$ fixed. Hence the columns of $\mathcal{H}^{(2)}$ are square summable for $\lambda>0$.

Second, consider the Jacobi matrix $\mathcal{J}$ determined  by the sequences
\[
 b_{n}:=0 \quad \mbox{ and } \quad a_{n}:=\frac{1}{2}\sqrt{(n+1)(n+2\lambda)},
\]
for $n\in\N_{0}$. By setting $\kappa_{n}=(-1)^{n}(n+2\lambda)/2$ in~\eqref{eq:commut_eq_reform}, one obtains the equation
\[
 \left(m-(-1)^{m+n}n\right)h_{m+n-1}-\left(m+2\lambda-(-1)^{m+n}(n+2\lambda)\right)h_{m+n+1}=0.
\]
When the parity of $m+n$ is distinguished, one arrives at difference equations entirely in $m+n$. 
This leads to two ordinary difference equations
\[
 h_{2k-1}-h_{2k+1}=0 \quad \mbox{ and } \quad (2k-1)h_{2k-2}-(4\lambda+2k-1)h_{2k}=0,
\]
where $k\in\N$. A solution reads
\begin{equation}
 h_{2k+1}=0 \quad \mbox{ and } \quad h_{2k}=\frac{\Gamma\left(k+1/2\right)}{\Gamma\left(2\lambda+k+1/2\right)},
\label{eq:h_n_hank2}
\end{equation}
for $k\in\N$. Further, it follows from~\eqref{eq:w_kappa} that the weights read 
\begin{equation}
 w_{n}=(-1)^{n(n-1)/2}\sqrt{\frac{\Gamma(n+2\lambda)}{n!}},
\label{eq:w_n_hank2}
\end{equation}
for $n\in\N_{0}$, where we chose $w_{0}:=\sqrt{\Gamma(2\lambda)}$. Using the obtained formulas~\eqref{eq:h_n_hank2} and~\eqref{eq:w_n_hank2},
$w_{n}w_{m}h_{m+n}$ coincides with $\mathcal{H}_{m,n}^{(2)}$ for all $m,n\in\N_{0}$. Hence $\mathcal{H}^{(2)}$ and $\mathcal{J}$ commute.

For $\lambda>0$, put
\begin{equation}
\hat{P}_{n}(x):=\left(\frac{\Gamma(2\lambda)\,n!}{\Gamma(n+2\lambda)}\right)^{\!1/2}P_{n}^{(\lambda)}\left(x;\frac{\pi}{2}\right)\!, \quad n\in\N_{0},
\label{eq:def_hatP_hank2}
\end{equation}
and $\hat{P}_{-1}(x):=0$, where $P_{n}^{(\lambda)}$ are the Meixner--Pollaczek polynomials~\eqref{eq:def_meix-polllac}.  The three-term recurrence relation for the Meixner--Pollaczek polynomials~\cite[Eq.~9.7.3]{koekoek-etal_10} yields
\[
 \sqrt{(n+1)(n+2\lambda)}\hat{P}_{n+1}(x)-2x\hat{P}_{n}(x)+\sqrt{n(n-1+2\lambda)}\hat{P}_{n-1}(x)=0,
\]
for $n\in\N_{0}$. Thus, the vector $(\hat{P}_{0}(x),\hat{P}_{1}(x),\hat{P}_{2}(x),\dots)$
is a formal eigenvector of the Jacobi matrix $\mathcal{J}$ corresponding to the eigenvalue $x$. Furthermore, since $\mathcal{J}$ is proper by Carleman's condition,
$\{\hat{P}_{n} \mid n\in\N_{0}\}$ is an orthonormal basis of $L^{2}(\R,\dd\mu)$ with
\begin{equation}
 \frac{\dd\mu}{\dd x}(x)=\frac{2^{2\lambda-1}}{\pi\Gamma(2\lambda)}\left|\Gamma(\lambda+\ii x)\right|^{2}, \quad x\in\R,
\label{eq:mu_hank2}
\end{equation}
where the measure $\mu$ was deduced from the orthogonality relation for the Meixner--Pollaczek polynomials~\cite[Eq.~9.7.2]{koekoek-etal_10}.

Using the formula for $h$ from Theorem~\ref{thm:def_H}, we get
\begin{equation}
h(x)=\sum_{n=0}^{\infty}\mathcal{H}_{n,0}^{(2)}\hat{P}_{n}(x)=w_{0}\sum_{n=0}^{\infty}w_{n}h_{n}\hat{P}_{n}(x)=
\Gamma(2\lambda)\sum_{n=0}^{\infty}\frac{(-1)^{n}\Gamma\left(n+1/2\right)}{\Gamma\left(n+2\lambda+1/2\right)}P_{2n}^{(\lambda)}\left(x;\frac{\pi}{2}\right)\!,
\label{eq:h_hank2_pre}
\end{equation}
where we have used formulas~\eqref{eq:h_n_hank2}, \eqref{eq:w_n_hank2}, and~\eqref{eq:def_hatP_hank2}. This time, the right-hand side of~\eqref{eq:h_hank2_pre} cannot be readily simplified
by using a generating function for the Meixner--Pollaczek polynomials. Nevertheless, it can be expressed in terms of the Gamma function as shows the next statement.

\begin{lem}\label{lem:h_hank2}
 For $x\in\R$, $\lambda>0$, and $h$ defined in~\eqref{eq:h_hank2_pre}, one has
 \begin{equation}
   h(x)=\frac{2^{2\lambda-1}}{\Gamma(2\lambda)}|\Gamma(\lambda+\ii x)|^{2}.
 \label{eq:h_hank2}
 \end{equation}
 Consequently, $h(\R)=(0,h(0)]$.
\end{lem}

\begin{proof}
For $\phi=\pi/2$, the generating function~\eqref{eq:gener_func_meix-pol} gets the form
\[
 \sum_{n=0}^{\infty}P_{n}^{(\lambda)}\left(x;\frac{\pi}{2}\right)t^{n}=(1-\ii t)^{-\lambda+\ii x}(1+\ii t)^{-\lambda-\ii x}.
\]
It follows that
\begin{equation}
 \sum_{n=0}^{\infty}(-1)^{n}P_{2n}^{(\lambda)}\left(x;\frac{\pi}{2}\right)\xi^{2n}=\Re g(\xi)
\label{eq:gener_func_meix-pol_g}
\end{equation}
for $\xi\in\R$, $|\xi|<1$, where
\[
 g(\xi):=(1+\xi)^{-\lambda+\ii x}(1-\xi)^{-\lambda-\ii x}.
\]

By using elementary properties of the Beta and the Gamma function, one obtains
\[
 \int_{-1}^{1}\xi^{2n}\left(1-\xi^{2}\right)^{2\lambda-1}\dd\xi=B\left(2\lambda,n+1/2\right)
 =\frac{\Gamma(2\lambda)\Gamma\left(n+1/2\right)}{\Gamma\left(n+2\lambda+1/2\right)},
\]
for $\lambda>0$ and $n\in\N_{0}$. Thus, recalling~\eqref{eq:h_hank2_pre}, we can multiply~\eqref{eq:gener_func_meix-pol_g} 
by $(1-\xi^{2})^{2\lambda-1}$ and integrate with respect to~$\xi$ from $-1$ to $1$ getting
\begin{equation}
 h(x)=\int_{-1}^{1}\!\left(1-\xi^{2}\right)^{2\lambda-1}\Re g(\xi)\dd\xi
 =2\Re\int_{0}^{1}\!\left(1+\xi\right)^{\lambda-1+\ii x}\left(1-\xi\right)^{\lambda-1-\ii x}\dd\xi.
\label{eq:h_integral_hankel2_in_proof}
\end{equation}
The interchange of the integral and the sum can be justified by Fubini's theorem. Indeed, for any $x\in\R$ and $\lambda>0$, there exists a constant $C_{\lambda,x}>0$ such that
\[
 \left|P_{n}^{(\lambda)}\left(x;\frac{\pi}{2}\right)\right|\leq C_{\lambda,x}n^{\lambda-1}
\]
which follows from the asymptotic behavior of the Meixner--Pollaczek polynomials
\[
 P_{n}^{(\lambda)}\left(x;\frac{\pi}{2}\right)=n^{\lambda-1}\Re\left(\frac{2^{1-\lambda-\ii x}}{\Gamma(\lambda-\ii x)}\ii^{n}n^{-\ii x}\right)(1+o(1)),\quad \mbox{ as } n\to\infty,
\]
see~\cite[p.~172]{ismail05}. Now, it is clear that
\[
 \sum_{n=0}^{\infty}\left|P_{2n}^{(\lambda)}\left(x;\frac{\pi}{2}\right)\right|\int_{-1}^{1}\xi^{2n}\left(1-\xi^{2}\right)^{2\lambda-1}\dd\xi
 \leq C_{x,\lambda}2^{\lambda-1}\Gamma(2\lambda)\sum_{n=0}^{\infty}n^{\lambda-1}\frac{\Gamma\left(n+\frac{1}{2}\right)}{\Gamma\left(n+2\lambda+\frac{1}{2}\right)}<\infty,
\]
since
\[
 \frac{\Gamma\left(n+\frac{1}{2}\right)}{\Gamma\left(n+2\lambda+\frac{1}{2}\right)}=n^{-2\lambda}(1+o(1)), \quad \mbox{ as } n\to\infty,
\]
by Stirling's formula.

Next, we further rewrite the integral on the right-hand side of~\eqref{eq:h_integral_hankel2_in_proof} by substituting for
$\xi=\tanh(t)$ that results in the expression
\[
 h(x)=2\int_{0}^{\infty}\frac{\cos(2xt)}{\cosh^{2\lambda}(t)}\dd t.
\]
The above integral can be simplified with the aid of the identity~\cite[Eq.~3.985 ad~1]{gradsteyn-ryzhik07}
\[
 \int_{0}^{\infty}\frac{\cos(at)}{\cosh^{\nu}(\beta t)}\dd t=\frac{2^{\nu-2}}{\beta\Gamma(\nu)}\Gamma\left(\frac{\nu}{2}+\frac{\ii a}{2\beta}\right)\Gamma\left(\frac{\nu}{2}-\frac{\ii a}{2\beta}\right)\!,
\]
where $\Re\beta>0$, $\Re\nu>0$, and $a>0$, which yields~\eqref{eq:h_hank2}.

Since, for all $\lambda>0$ and $x\in\R$, one has~\cite[Eq.~5.6.6]{dlmf}
\[
 0<|\Gamma(\lambda+\ii x)|\leq\Gamma(\lambda)
\]
and
\[
 \lim_{x\to\pm\infty}\Gamma(\lambda+\ii x)=0
\]
by~\eqref{eq:gamma_asympt_imag}, it follows that $h(\R)=(0,h(0)]$.
\end{proof}

Formulas~\eqref{eq:def_hatP_hank2}, \eqref{eq:mu_hank2} and~\eqref{eq:h_hank2} are those stated in Theorem~\ref{thm:meixpol} which were to be proven.
It remains to show that the matrix $\mathcal{H}^{(2)}$ is proper. To this end, we may again verify the condition~\eqref{eq:int_cond_H_proper} for $\mu$ and $h$ given 
by~\eqref{eq:mu_hank2} and~\eqref{eq:h_hank2}. The verification is a straightforward use of the asymptotic formula~\eqref{eq:gamma_asympt_imag}.

It is a consequence of the formula~\eqref{eq:mu_hank2} and Lemma~\ref{lem:h_hank2} that the operator $H^{(2)}$ has the spectrum
\[
 \spec\left(H^{(2)}\right)=\spec_{\text{ac}}\left(H^{(2)}\right)=\overline{h(\R)}=[0,h(0)],
\]
where
\[
 h(0)=\|H^{(2)}\|=2^{2\lambda-1}\frac{\Gamma^{2}(\lambda)}{\Gamma(2\lambda)}=\frac{\sqrt{\pi}{\Gamma(\lambda)}}{\Gamma(\lambda+1/2)}.
\]
The last equality holds due to the identity~\cite[Eq.~5.5.5]{dlmf}
\[
 \Gamma(2z)=\frac{2^{2z-1}}{\sqrt{\pi}}\Gamma(z)\Gamma(z+1/2).
\]

\subsubsection{Proof of Corollary~\ref{cor:hank2eo}}
Since $\mathcal{H}_{m,n}^{(2)}=0$ if $m+n$ is odd, the orthogonal decomposition 
\[
\ell^{2}(\N_{0})=\ell^{2}(2\N_{0})\oplus\ell^{2}(2\N_{0}+1)
\]
is $H^{(2)}$-invariant. Consider another orthogonal decomposition of $L^{2}(\R,\dd\mu)$ as an orthogonal sum of the subspaces of even and odd functions,
\[
 L^{2}(\R,\dd\mu)=L^{2}_{\text{even}}(\R,\dd\mu)\oplus L^{2}_{\text{odd}}(\R,\dd\mu).
\]
Note that the unitary mapping $U:\ell^{2}(\N_{0})\to L^{2}(\R,\dd\mu): e_{n}\mapsto\hat{P}_{n}$ maps $\ell^{2}(2\N_{0})$ onto $L^{2}_{\text{even}}(\R,\dd\mu)$ and $\ell^{2}(2\N_{0}+1)$ onto $L^{2}_{\text{odd}}(\R,\dd\mu)$ since the polynomial $\hat{P}_{n}$ is even or odd if and only if the index $n$ is even or odd, respectively. Moreover, one can make use of the natural unitary equivalence of both $L^{2}_{\text{even}}(\R,\dd\mu)$ and $L^{2}_{\text{odd}}(\R,\dd\mu)$ with $L^{2}((0,\infty),\dd\mu)$. Restricting $H^{(2)}$ to $\ell^{2}(2\N_{0})$ and $\ell^{2}(2\N_{0}+1)$ yields Corollary~\ref{cor:hank2eo}.

\subsection{Proof of Theorem~\ref{thm:hermite} and Corollary~\ref{cor:hank3eo}}

First, an application of the Stirling formula implies
\[
\mathcal{H}_{m,n}^{(3)}=(-1)^{m(m-1)/2+n(n-1)/2}\frac{\pi^{1/4}}{\sqrt{n!}}2^{-\frac{m+n}{2}+\frac{1}{4}}m^{\frac{n}{2}-\frac{1}{4}}, \quad \mbox{ as } m\to\infty,
\]
for $m+n$ even and $n\in\N_{0}$ fixed. Consequently, columns of $\mathcal{H}^{(2)}$ represent square summable sequences.

Define~$\mathcal{J}$ by putting
\[
 b_{n}:=0 \quad \mbox{ and } \quad a_{n}:=\sqrt{(n+1)/2},
\]
for $n\in\N_{0}$. With this setting and putting $\kappa_{n}:=(-1)^{n}/\sqrt{2}$ in~\eqref{eq:commut_eq_reform}, one arrives at the equation
\[
 \left(m-(-1)^{m+n}n\right)h_{m+n-1}-\left(1-(-1)^{m+n}\right)h_{m+n+1}=0.
\]
Distinguishing between the parities of $m+n$, one obtains difference equations entirely in $m+n$ that takes the form
\[
 h_{2k-1}=0 \quad \mbox{ and } \quad (2k-1)h_{2k-2}-2h_{2k}=0,
\]
where $k\in\N$. A solution is
\begin{equation}
 h_{2k+1}=0 \quad \mbox{ and } \quad h_{2k}=\Gamma\left(k+1/2\right),
\label{eq:h_n_hank3}
\end{equation}
for $k\in\N$. Next, by using~\eqref{eq:w_kappa}, one gets 
\begin{equation}
 w_{n}=(-1)^{n(n-1)/2}\frac{1}{\sqrt{n!}},
\label{eq:w_n_hank3}
\end{equation}
for $n\in\N_{0}$, where we took $w_{0}:=1$. The resulting weighted Hankel matrix with elements $w_{n}w_{m}h_{m+n}$ coincides with $\mathcal{H}^{(3)}$ and thus 
$\mathcal{H}^{(3)}$ commutes with $\mathcal{J}$.

By using the three-term recurrence for Hermite polynomials~\cite[Eq.~9.15.3]{koekoek-etal_10}, one verifies that
\[
 \sqrt{n+1}\hat{P}_{n+1}(x)-\sqrt{2}x\hat{P}_{n}(x)+\sqrt{n}\hat{P}_{n-1}(x)=0, \quad n\in\N_{0},
\]
where
\begin{equation}
\hat{P}_{n}(x):=\frac{1}{\sqrt{2^{n}n!}}H_{n}(x), \quad n\in\N_{0},
\label{eq:def_hatP_hank3}
\end{equation}
and $\hat{P}_{-1}(x):=0$. Consequently, the vector $(\hat{P}_{0}(x),\hat{P}_{1}(x),\hat{P}_{2}(x),\dots)$
is a formal eigenvector of the Jacobi matrix $\mathcal{J}$ corresponding to the eigenvalue $x$. In addition, $\{\hat{P}_{n} \mid n\in\N_{0}\}$ 
is an orthonormal basis of $L^{2}(\R,\dd\mu)$, where
\begin{equation}
 \frac{\dd\mu}{\dd x}(x)=\frac{1}{\sqrt{\pi}}e^{-x^{2}}, \quad x\in\R,
\label{eq:mu_hank3}
\end{equation}
is the density of the measure of orthogonality for the Hermite polynomials~\cite[Eq.~9.15.2]{koekoek-etal_10}.

As far as the function $h$ is concerned, we have
\[
h(x)=\sum_{n=0}^{\infty}\mathcal{H}^{(3)}_{0,n}\hat{P}_{n}(x)=w_{0}\sum_{n=0}^{\infty}w_{n}h_{n}\hat{P}_{n}(x)=\sum_{n=0}^{\infty}\,\frac{(-1)^{n}}{2^{n}(2n)!}\,\Gamma\!\left(n+\frac{1}{2}\right)\! H_{2n}(x),
\]
where the formulas~\eqref{eq:h_n_hank3}, \eqref{eq:w_n_hank3}, and~\eqref{eq:def_hatP_hank3} were used. 
Since $H_{n}(-x)=(-1)^{n}H_{n}(x)$ and $\Gamma(n+1/2)=\sqrt{\pi}\,2^{-2n}(2n)!/n!$, the above expression 
for $h$ can be rewritten as
\begin{equation}
h(x)=\sqrt{\pi}\,\sum_{n=0}^{\infty}\frac{(-1)^{n}}{2^{3n}n!}\, H_{2n}(x)\,=\,\frac{\sqrt{\pi}}{2}\,\sum_{n=0}^{\infty}\frac{\ii^{n}}{2^{3n/2}\,[n/2]!}\big(H_{n}(x)+H_{n}(-x)\big)
 =\sqrt{2\pi}e^{-x^{2}},
\label{eq:h_hank3}
\end{equation}
where $[n/2]$ denotes the floor of $n/2$. The last equality holds due to the generating function formula~\cite[Eq.~9.15.14]{koekoek-etal_10}
\[
\sum_{n=0}^{\infty}\frac{t^{n}}{[n/2]!}\, H_{n}(x)=\frac{1+2xt+4t^{2}}{(1+4t^{2})^{3/2}}\exp\!\left(\frac{4x^{2}t^{2}}{1+4t^{2}}\right)\!.
\]

The obtained formulas~\eqref{eq:def_hatP_hank3}, \eqref{eq:mu_hank3}, and~\eqref{eq:h_hank3} for~$\hat{P}_{n}$, $h$, and~$\mu$ coincide with those stated in Theorem~\ref{thm:hermite}. The verification of the 
condition~\eqref{eq:int_cond_H_proper} is immediate in this case and hence $\mathcal{H}^{(3)}$ determines a unique operator~$H^{(3)}$. Furthermore, one deduces 
from~\eqref{eq:h_hank3} and~\eqref{eq:mu_hank3} that
\[
\spec\left(H^{(3)}\right)=\spec_{\text{ac}}\left(H^{(3)}\right)=\overline{h(\R)}=[\,0,\sqrt{2\pi}\,] \quad \mbox{ and } \quad \|H^{(3)}\|=\sqrt{2\pi}\,.
\]
Finally, the proof of Corollary~\ref{cor:hank3eo} is completely analogous to the proof of Corollary~\ref{cor:hank2eo}.

\subsection{Proof of Theorem~\ref{thm:dual_hahn} and Corollary~\ref{cor:dual_hahn}}

Let $N\in\N_{0}$ and $J$ be the $(N+1)\times(N+1)$ Jacobi matrix with diagonal and off-diagonal elements given by
\begin{equation}
 b_{n}=n(\delta+N+1-n)+(N-n)(n+\gamma+1)
\label{eq:b_n_dual_hahn}
\end{equation}
and
\begin{equation}
 a_{n}=-\sqrt{(n+1)(n+1+\gamma)(N-n)(N-n+\delta)}, 
\label{eq:a_n_dual_hahn}
\end{equation}
for $n\in\{0,1,\dots,N\}$. The parameters are restricted so that both $\gamma>-1$ and $\delta>-1$.

We again use the equation~\eqref{eq:commut_eq_reform} to find a weighted Hankel matrix commuting with~$J$
this time with the choice 
\begin{equation}
\kappa_{n}:=(N-n)(N-n+\delta). 
\label{eq:kappa_n_dual_hahn}
\end{equation}
It is easy to see that the equation~\eqref{eq:commut_eq_reform} is still equivalent to the commutativity of the respective finite $(N+1)\times(N+1)$
matrices if it holds for all $0\leq m<n\leq N$ and, importantly, $a_{-1}$ as well as $\kappa_{N}$ vanish. Clearly, with choices~\eqref{eq:a_n_dual_hahn}
and~\eqref{eq:kappa_n_dual_hahn}, $a_{-1}=\kappa_{N}=0$.

Substituting from~\eqref{eq:b_n_dual_hahn}, \eqref{eq:a_n_dual_hahn}, and~\eqref{eq:kappa_n_dual_hahn} into~\eqref{eq:commut_eq_reform}, one obtains a difference 
equation entirely in the variable $m+n$ after canceling the factor $m-n$ common to all the coefficients. Putting $k:=m+n$, the resulting equation reads
\[
 (2k-\delta+\gamma-2N)h_{k}-(k+\gamma)h_{k-1}-(k-2N-\delta)h_{k+1}=0,
\]
for $k\in\N_{0}$. The non-constant solutions is
\begin{equation}
h_{n}=(-1)^{n}(1+\gamma)_{n}(1+\delta)_{2N-n}, \quad n\in\{0,1,\dots,2N\}.
\label{eq:h_n_hank4}
\end{equation}
According to~\eqref{eq:w_kappa}, the weights are
\begin{equation}
 w_{n}=\frac{(-1)^{n}}{\sqrt{n!(N-n)!(1+\gamma)_{n}(1+\delta)_{N-n}}}, \quad n\in\{0,1,\dots,N\},
\label{eq:w_n_hank4}
\end{equation}
where we put $w_{0}:=1/\sqrt{N!(1+\delta)_{N}}$. The weighted Hankel matrix with elements $w_{m}w_{n}h_{m+n}$ coincides with $H^{(4)}$. Hence $H^{(4)}$ and~$J$ commute.

Recall the definition of the dual Hahn polynomials~\eqref{eq:def_dualhahn}, \eqref{eq:def_lam_dualhahn}, and put 
\begin{equation}
\hat{P}_{n}(x):=\sqrt{\frac{N!(1+\gamma)_{n}(1+\delta)_{N-n}}{n!(N-n)!(1+\delta)_{N}}}R_{n}\left(\lambda(x);\gamma,\delta,N\right)\!,
\label{eq:def_hatP_hank4}
\end{equation}
for $n\in\{0,1,\dots,N\}$. It follows from the three-term recurrence for the dual Hahn polynomials~~\cite[Eq.~9.6.3]{koekoek-etal_10} that
the vector $(\hat{P}_{0}(x),\hat{P}_{1}(x),\dots,\hat{P}_{N}(x))$
is the eigenvector of $J$ corresponding to the eigenvalue $\lambda(x)$, for $x\in\{0,1,\dots,N\}$.
At the same time, $\{\hat{P}_{n}\mid 0\leq n\leq N\}$ is an orthogonal basis of the Hilbert space $L^{2}(\R,\dd\mu)$,
where
\begin{equation}
 \mu=(1+\delta)_{N}N!\sum_{x=0}^{N}\frac{(2x+\gamma+\delta+1)(1+\gamma)_{x}}{(1+x+\gamma+\delta)_{N+1}(1+\delta)_{x}(N-x)!x!}\delta_{x}.
\label{eq:mu_hank4}
\end{equation}
for $\gamma,\delta>-1$; see~\cite[Eq.~9.6.2]{koekoek-etal_10}.

Consequently, the unitary mapping $U:\C^{N+1}\to L^{2}(\R,\dd\mu)$ defined by $Ue_{n}:=\hat{P}_{n}$, for $n\in\{0,1,\dots,N\}$,
diagonalizes $J$. Concretely, $UJU^{-1}$ acts as the multiplication operator by $\lambda(x)$ on $L^{2}(\R,\dd\mu)$. 
Since the spectrum of $J$ is simple and $H^{(4)}$ commutes with $J$, $H^{(4)}$ acts as the multiplication operator by a function $h$ on $L^{2}(\R,\dd\mu)$, where
\[
 h(x)=h(x)\hat{P}_{0}(x)=UH^{(4)}e_{0}=\sum_{n=0}^{N}H^{(4)}_{n,0}\hat{P}_{n}(x)=w_{0}\sum_{n=0}^{N}w_{n}h_{n}\hat{P}_{n}(x).
\]
Referring to formulas~\eqref{eq:h_n_hank4}, \eqref{eq:w_n_hank4}, and~\eqref{eq:def_hatP_hank4}, one gets
\begin{equation}
h(x)=\frac{1}{(1+\delta)_{N}}\sum_{n=0}^{N}\frac{(1+\gamma)_{n}(1+\delta)_{2N-n}}{n!(N-n)!}R_{n}\left(\lambda(x);\gamma,\delta,N\right).
\label{eq:h_hankl4_pre}
\end{equation}
Surprisingly, the above expression can be considerably simplified if $x\in\{0,1,\dots,N\}$.

\begin{lem}
 For $\gamma,\delta>-1$, $x\in\{0,1,\dots,N\}$, and $h$ defined by~\eqref{eq:h_hankl4_pre}, one has 
 \begin{equation}
 h(x)=\binom{2N+\gamma+\delta+1}{N-x}.
 \label{eq:h_hankl4}
 \end{equation}
\end{lem}

\begin{proof}
With the aid of the identity
 \[
  \frac{(1+\delta)_{2N-n}}{(1+\delta)_{N}}=(-1)^{N+n}\frac{(n-2N-\delta)_{N}}{(-N-\delta)_{n}}, \quad n\in\{0,1,\dots,N\},
 \]
one can rewrite the formula~\eqref{eq:h_hankl4_pre} as
\begin{equation}
  h(x)=\frac{(-1)^{N}}{N!}\frac{\dd^{N}}{\dd t^{N}}\bigg|_{t=1}\Phi(t),
  \label{eq:h_eq_der_Phi}
\end{equation}
where
\[
 \Phi(t):=\sum_{n=0}^{N}\frac{(-N)_{n}(\gamma+1)_{n}}{n!(-N-\delta)_{n}}R_{n}\left(\lambda(x);\gamma,\delta,N\right)t^{n-N-\delta-1}.
\]
Since, for $x=0,1,\dots,N$, it holds~\cite[Eq.~9.6.12]{koekoek-etal_10}
\[
 \sum_{n=0}^{N}\frac{(-N)_{n}(\gamma+1)_{n}}{n!(-N-\delta)_{n}}R_{n}\left(\lambda(x);\gamma,\delta,N\right)t^{n}=(1-t)^{x}\pFq{2}{1}{x-N,x+\gamma+1}{-\delta-N}{t},
\]
we get
\[
 \Phi(t)=t^{-N-\delta-1}(1-t)^{x}\pFq{2}{1}{x-N,x+\gamma+1}{-\delta-N}{t}.
\]

Further, by differentiation and making use of the definition of the Gauss hypergeometric function, we obtain
\begin{align*}
 \frac{\dd^{N}}{\dd t^{N}}\bigg|_{t=1}\Phi(t)&=\frac{(-1)^{x}N!}{(N-x)!}\frac{\dd^{N-x}}{\dd t^{N-x}}\bigg|_{t=1}t^{-N-\delta-1}\pFq{2}{1}{x-N,x+\gamma+1}{-\delta-N}{t}\\
 &=(-N)_{x}\sum_{k=0}^{N-x}\frac{(x-N)_{k}(x+\gamma+1)_{k}(k-2N-\delta+x)_{N-x-k}}{k!}\\
 &=(-N)_{x}(x-2N-\delta)_{N-x}\,\pFq{2}{1}{x-N,x+\gamma+1}{x-2N-\delta}{1},
\end{align*}
for $x=0,1,\dots,N$. The above Gauss hypergeometric function can be evaluated by the Chu--Vandermonde identity~\cite[Eq.~15.4.24]{dlmf}
\[
 \pFq{2}{1}{-n,b}{c}{1}=\frac{(c-b)_{n}}{(c)_{n}}, \quad n\in\N_{0},
\]
which results in the formula
\begin{equation}
\frac{\dd^{N}}{\dd t^{N}}\bigg|_{t=1}\Phi(t)=(-N)_{x}(-2N-\gamma-\delta-1)_{N-x},
 \label{eq:der_Phi_eval}
\end{equation}
for $x\in\{0,1,\dots,N\}$. The combination of~\eqref{eq:h_eq_der_Phi} and~\eqref{eq:der_Phi_eval} yields the statement.
\end{proof}

The formulas~\eqref{eq:def_hatP_hank4}, \eqref{eq:mu_hank4} and~\eqref{eq:h_hankl4} are those stated in Theorem~\ref{thm:dual_hahn}. Moreover, 
since $h(x)$, for $x\in\{0,1,\dots,N\}$, are eigenvalues of $H^{(4)}$, the formula~\eqref{eq:h_hankl4} completes the proof of Theorem~\ref{thm:dual_hahn}.

As far as the proof of Corollary~\ref{cor:dual_hahn} is concerned, it suffices to note that the Hankel matrix $G$ is the ``Hankel part'' of $H^{(4)}$ meaning 
that $H^{(4)}=WGW$, where $W=\diag(w_{0},w_{1},\dots w_{N})$. Consequently,
one has
\[
 \det G = \left(\prod_{n=0}^{N}w_{n}^{-2}\right)\left(\prod_{x=0}^{N}h(x)\right)\!.
\]
By using formulas~\eqref{eq:w_n_hank4} and \eqref{eq:h_hankl4} and making simple manipulations, 
one derives the formula from the statement of Corollary~\ref{cor:dual_hahn} for $\gamma,\delta>-1$.
Since both sides of the resulting identity are polynomials in $\gamma$ and $\delta$, the equality has to remain true for all $\gamma,\delta\in\C$.

\section{Consequences for special functions} \label{sec:conseq_spec_func}

From the obtained results, one can derive certain identities for the involved orthogonal polynomials or alternatively the hypergeometric functions.
Some of them are known, some seem to be new. We list these identities in this section.

In all the studied cases, we have an operator $H$ on $\ell^{2}(\N)$ together with a unitary mapping 
\[
 U:\ell^{2}(\N_{0})\to L^{2}(\R,\dd\mu):e_{n}\to\hat{P}_{n},
\]
such that $UHU^{-1}$ is the multiplication operator by a function $h$ on $L^{2}(\R,\dd\mu)$. Consequently, one has
\begin{equation}
 H_{m,n}=\langle e_{m}, He_{n}\rangle_{\ell^{2}(\N_{0})}=\langle \hat{P}_{m}, h\hat{P}_{n}\rangle_{L^{2}(\R,\dd\mu)}=\int_{\R}h(x)\hat{P}_{m}(x)\hat{P}_{n}(x)\dd\mu(x),
\label{eq:int_ident_gener}						
\end{equation}
for $m,n\in\N_{0}$.

We go through the list of studied weighted Hankel matrices and substitute in~\eqref{eq:int_ident_gener} for the respective matrix element on the left-hand side as well as
$h$, $\mu$, and $\hat{P}_{n}$ on the right-hand side. For example, making use of~\eqref{eq:def_hank1} and~\eqref{eq:def_hatP_hank1_meixpol}, \eqref{eq:mu_hank1_meixpol}, 
\eqref{eq:h_hank1_meixpol} and also~\eqref{eq:param_iden_meixpol} in~\eqref{eq:int_ident_gener}, one arrives at the following integral formula for Meixner--Pollaczek 
polynomials after some simple manipulations:
\begin{equation}
 \int_{\R}e^{(4\phi-\pi)x}P_{m}^{(\lambda)}(x;\phi)P_{n}^{(\lambda)}(x;\phi)|\Gamma(\lambda+\ii x)|^{2}\dd x =
 \frac{\pi\Gamma(m+n+2\lambda)}{2^{m+n+2\lambda-1}\sin^{2\lambda}(2\phi)\cos^{m+n}(\phi)m!n!},
\label{eq:int_ident_hank1_meixpol}
\end{equation}
for $m,n\in\N_{0}$, $\lambda>0$ and $0<\phi<\pi/3$. Since the integral on the left-hand side of~\eqref{eq:int_ident_hank1_meixpol} converges for $0<\phi<\pi/2$ 
and both sides are analytic in $\phi$ on this interval, the identity~\eqref{eq:int_ident_hank1_meixpol} remains true for all $\phi\in(0,\pi/2)$.

Similarly, using~\eqref{eq:def_hank1} with $k=1/2$ and~\eqref{eq:def_hatP_hank1_laguerre}, \eqref{eq:mu_hank1_laguerre}, \eqref{eq:h_hank1_laguerre} 
in~\eqref{eq:int_ident_gener}, one gets the integral formula for Laguerre polynomials
\begin{equation}
 \int_{0}^{\infty}L_{m}^{(\alpha)}(x)L_{n}^{(\alpha)}(x)x^{\alpha}e^{-2x}\dd x =
 \frac{\Gamma(m+n+\alpha+1)}{2^{m+n+\alpha+1}m!n!},
\label{eq:int_ident_hank1_laguerre}
\end{equation}
for $m,n\in\N_{0}$ and $\alpha>-1$. The identity~\eqref{eq:int_ident_hank1_laguerre} is actually known in a more general form; see~\cite[Eq.~7.414 ad~4]{gradsteyn-ryzhik07} (with
$\lambda=\mu=1$, $b=2$). Next, substituting from ~\eqref{eq:def_hank1} and~\eqref{eq:def_hatP_hank1_meixner}, \eqref{eq:mu_hank1_meixner}, \eqref{eq:h_hank1_meixner} and 
also~\eqref{eq:param_iden_meixner} in~\eqref{eq:int_ident_gener}, one obtains the summation formula for Meixner polynomials
\[
 \sum_{x=0}^{\infty}\frac{(\beta)_{x}}{x!}c^{2x}M_{m}(x;\beta,c)M_{n}(x;\beta,c)
 =\frac{(\beta)_{m+n}}{(1-c)^{\beta}(1+c)^{m+n+\beta}(\beta)_{m}(\beta)_{n}},
\]
for $m,n\in\N_{0}$, $\beta>0$ and $c\in(0,1)$.

Further, \eqref{eq:def_hank2}, \eqref{eq:def_hatP_hank2}, \eqref{eq:mu_hank2}, \eqref{eq:h_hank2} used in~\eqref{eq:int_ident_gener} yields
\begin{align}
 \int_{\R}&P_{m}^{(\lambda)}\left(x;\frac{\pi}{2}\right)P_{n}^{(\lambda)}\left(x;\frac{\pi}{2}\right)|\Gamma(\lambda+\ii x)|^{4}\dd x \nonumber\\
 &\hskip68pt=(-1)^{\left[\frac{m+1}{2}\right]+\left[\frac{n+1}{2}\right]}\,\frac{\pi\,\Gamma(2\lambda)\Gamma(2\lambda+m)\Gamma(2\lambda+n)\Gamma((m+n+1)/2)}{4^{2\lambda-1}\,m!\,n!\,\Gamma(2\lambda+(m+n+1)/2)},
 \label{eq:int_ident_hank2}
\end{align}
for $m,n\in\N_{0}$ such that $m+n$ is even and $\lambda>0$. Since the parity of the polynomial $P_{n}^{(\lambda)}\left(\,\cdot\,;\pi/2\right)$ equals the parity of the index $n$,
the integral on the left-hand side of~\eqref{eq:int_ident_hank2} is clearly vanishing if $m+n$ is odd.

The third Hankel matrix~\eqref{eq:def_hank3} and formulas~\eqref{eq:def_hatP_hank3}, \eqref{eq:mu_hank3} and~\eqref{eq:h_hank3} used in~\eqref{eq:int_ident_gener}
yields the integral formula for Hermite polynomials
\begin{equation}
 \int_{\R}e^{-2x^{2}}H_{m}(x)H_{n}(x)\dd x=(-1)^{\left[\frac{m+1}{2}\right]+\left[\frac{n+1}{2}\right]}\,2^{(m+n-1)/2}\Gamma\left(\frac{m+n+1}{2}\right)\!,
 \label{eq:int_ident_hank3}
\end{equation}
for $m,n\in\N_{0}$ such that $m+n$ is even. The integral in~\eqref{eq:int_ident_hank3} obviously vanishes if $m+n$ is odd.
The identity~\eqref{eq:int_ident_hank3} is known, see~\cite[Eq.~7.374 ad~2]{gradsteyn-ryzhik07}.

Lastly, the finite weighted Hankel matrix~\eqref{eq:def_hank4} and corresponding formulas~\eqref{eq:def_hatP_hank4}, \eqref{eq:mu_hank4} and~\eqref{eq:h_hankl4} used in~\eqref{eq:int_ident_gener}
lead to the following identity for dual Hahn polynomials:
\begin{align*}
 \sum_{x=0}^{N}&\binom{2N+\gamma+\delta+1}{N-x}\frac{(2x+\gamma+\delta+1)(1+\gamma)_{x}}{(1+x+\gamma+\delta)_{N+1}(1+\delta)_{x}(N-x)!x!}\\
 &\times R_{m}(\lambda(x);\gamma,\delta,N)R_{n}(\lambda(x);\gamma,\delta,N)
 =\frac{(1+\gamma)_{m+n}(1+\delta)_{2N-m-n}}{(N!)^{2}(1+\gamma)_{m}(1+\gamma)_{n}(1+\delta)_{N-m}(1+\delta)_{N-n}},
\end{align*}
for $m,n\in\{0,1,\dots,N\}$, $N\in\N_{0}$ and $\gamma,\delta>-1$, where $\lambda(x)=x(x+\gamma+\delta+1)$.

Yet another corollary of Theorem~\ref{thm:dual_hahn} yields quite interesting identities with binomial coefficients. Namely, denoting
$\lambda_{0},\lambda_{1},\dots\lambda_{N}$ the eigenvalues of $H^{(4)}$, the equality
\begin{equation}
 \Tr H^{(4)}=\sum_{x=0}^{N}\lambda_{x}
\label{eq:trace_matr_eigen}
\end{equation}
implies the identity
\[
 \sum_{m=0}^{N}\binom{2m+\gamma}{m}\binom{2N-2m+\delta}{N-m}=\sum_{x=0}^{N}\binom{2N+1+\gamma+\delta}{x},
\]
for $N\in\N_{0}$ and $\gamma,\delta>-1$. Second, the equality 
\begin{equation}
 \Tr\left(H^{(4)}\right)^{\!2}=\sum_{x=0}^{N}\lambda_{x}^{2}
\label{eq:trace_matr_eigen_square}
\end{equation}
means that
\begin{align*}
 \sum_{m,n=0}^{N}\binom{m+n+\gamma}{m}\binom{m+n+\gamma}{n}\binom{2N-m-n+\delta}{N-m}&\binom{2N-m-n+\delta}{N-n}\\
 &\hskip48pt=\sum_{x=0}^{N}\binom{2N+1+\gamma+\delta}{x}^{\! 2}\!,
\end{align*}
for $N\in\N_{0}$ and $\gamma,\delta>-1$. 

Finally, let us remark that the infinite analogues of the formulas~\eqref{eq:trace_matr_eigen} and~\eqref{eq:trace_matr_eigen_square}, i.e.,
the Lidskii theorem and a formula for the Hilbert--Schmidt norm, hold for trace class operators and hence can be also applied to $H^{(1)}$ in the case when $0<k<1/2$.
However, by using the spectral properties summarized in Theorem~\ref{thm:meixpol_lag_meix}, case (iii), one obtains nothing but special cases of the well-known 
identity~\cite[Eq.~15.4.18]{dlmf}
\[
 \pFq{2}{1}{a,a+\frac{1}{2}}{2a}{z}=\frac{1}{\sqrt{1-z}}\left(\frac{1}{2}+\frac{1}{2}\sqrt{1-z}\right)^{1-2a}.
\]

\section*{Acknowledgement}
The authors acknowledge financial support by the Ministry of Education, Youth and Sports of the Czech Republic 
project no. CZ.02.1.01/0.0/0.0/16\_019/0000778.

\end{document}